\documentclass[11pt]{amsart}
\usepackage[margin=2.5cm]{geometry}
\usepackage{amsmath,amssymb,amsfonts,amsthm,amsfonts,mathrsfs,amsrefs}
\usepackage{amsfonts}
\usepackage{epic}
\usepackage{amsopn,amscd, graphicx}
\usepackage{color, transparent} 
\usepackage[usenames, dvipsnames]{xcolor}	
\usepackage[colorlinks, breaklinks,
bookmarks = false,
linkcolor = Blue,
urlcolor = Green,
citecolor = Green,
hyperfootnotes = false
]{hyperref}
\usepackage{palatino, mathpazo}
\usepackage{multirow}
\usepackage[all,cmtip]{xy}
\usepackage{stmaryrd}
 1
 1
 1

\newtheorem{theorem}{Theorem}[section]
\newtheorem{corollary}{Corollary}[section]
\newtheorem{proposition}{Proposition}[section]
\newtheorem{assumption}{Assumption}[section]

\theoremstyle{definition}
\newtheorem{definition}{Definition}[section]

\numberwithin{equation}{section}
\title{Asymptotics of torus equivariant Szeg\H{o} kernel on a compact CR manifold}
\author[W.-C~Shen]{Wei-Chuan Shen}
\email{shen1993@gate.sinica.edu.tw}
\address{Institute of Mathematics, Academia Sinica, 6F, Astronomy-Mathematics Building, No. 1, Sec. 4, Roosevelt Road, Taipei 10617, TAIWAN}
\date{}

\begin{document}
\begin{abstract}
For a compact CR manifold $(X,T^{1,0}X)$ of dimension $2n+1$, $n\geq 2$, admitting a $S^1\times T^d$ action, if the lattice point $(-p_1,\cdots,-p_d)\in\mathbb{Z}^{d}$ is a regular value of the
associate CR moment map $\mu$, then we establish the asymptotic expansion of the torus
equivariant Szeg\H{o} kernel $\Pi^{(0)}_{m,mp_1,\cdots,mp_d}(x,y)$ as $m\to +\infty$ under certain assumptions of the positivity of Levi
form and the torus action on $Y:=\mu^{-1}(-p_1,\cdots,-p_d)$.
\end{abstract}

\maketitle 
\tableofcontents
\section{Introduction and statement of the main results}
Let $(X,T^{1,0}X)$ be a Cauchy--Riemann (CR for short) manifold of dimension $2n+1$, and
$\Box^{(q)}_b$ be Kohn Laplacian for $(0,q)$-forms on $X$. The orthogonal projection
$\Pi^{(q)}:L^2_{(0,q)}(X)\to\ker\Box_{b}^{(q)}$ is called the Szeg\H{o} projection, and we call
its distributional kernel $\Pi^{(q)}(x,y)$ the Szeg\H{o} kernel. The study of the Szeg\H{o}
kernel is a classical subject in several complex variables and CR geometry. For example, when $X$
is the boundary of a strongly pseudoconvex domain in $\mathbb{C}^n$, $n\geq 2$, which implies that $X$ is a
strongly pseudoconvex CR manifold, Boutet de Monvel and Sj\"ostrand \cite{BS1976} proved
that when $q=0$, $\Pi^{(0)}(x,y)$ is a Fourier integral operator with complex valued phase
function. This kind of description of kernel function has profound impact in many aspects, such
as spectral theory for Toeplitz operator, geometric quantization and K\"ahler geometry
\cites{HsiaoMarinsecu2014,HsiaoMarinescu2017berezin,MaMarinescu2007,Zelditch1998,BG1981,G1995}. 

In some recent progress \cites{HsiaoHuang2017,FritschHermannHsiao2018,HsiaoMaMarinescu}, people
start
to consider CR manifolds with Lie group action $G$. The study of $G$-equivariant Szeg\H{o} kernels
is closely related to the problems of equivariant CR
embedding and geometric quantization on CR manifolds. The goal of this paper is especially to
understand the asymptotic behavior of torus equivariant Szeg\H{o} kernel. Let us briefly explain
out motivation. Within the manifolds drastically studied on this
topic,  Sasakian manifold, which is a compact, strongly pseudoconvex and torsion free CR
manifold, stands for the odd-dimensional
counter part in K{\" a}hler geometry and serves as a significant example. We say a CR manifold
$(X,T^{1,0}X)$ is a quasi-regular Sasakian manifold if it admits a CR and transversal circle action. If $X$ is
quasi-regular, in \cite{HermannHsiaoLi2017expansion} they showed that  $\dim H^0_{b,m}(X)\approx
m^n$ and the Szeg\H{o} kernel for $H^0_{b,m}(X)$ admits a full asymptotic expansion as $m\to
+\infty$, where $H^0_{b,m}(X)$ is the space of $m$-th Fourier component with respect to the circle action of the global CR
functions on $X$. We say a CR manifold $(X,T^{1,0}X)$ is an irregular Sasakian manifold if it endows with a CR transversal
$\mathbb{R}$-action, which does not come from any circle action. Suppose now $X$ is irregular and $T$ be the fundamental vector field of the
$\mathbb{R}$-action. Take a $\mathbb{R}$-invariant $L^2$-inner product on $X$ and consider the weak maximal extension of $T$ on $L^2$ functions, then in \cite{HermannHsiaoLi2017embedding} it was shown that $T$ is a self-adjoint operator, and the spectrum of $T$, denoted by
${\rm Spec}(T)$, is a countable subset in $\mathbb{R}$. Moreover, all the spectrum of $T$ are eigenvalues. On irregular Sasakian
manifolds, it is important to understand the space
$H_{b,\alpha}^0(X):=\{u\in\mathscr{C}^\infty(X):\overline{\partial}_b u=0, Tu=i\alpha u\}$, where
$\overline{\partial}_b$ denotes the tangential Cauchy--Riemann operator on $X$. Different from the
quasi-regular situation, in general it is very difficult to see which $\alpha\in{\rm Spec}(T)$
makes $\dim H_{b,\alpha}^0(X)>0$. It is revealed in \cite{HermannHsiaoLi2018} that if we sum
over $\alpha$ between $0$ and $k$ then the weigted Szeg\H{o} kernel for the space
$\bigoplus_{\substack{\alpha\in{\rm Spec}(T)\\0\leq\alpha\leq k}} H^0_{b,\alpha}(X)$ admits an
asymptotic expansion in $k$ as $k\to+\infty$. Accordingly, there are many $\alpha\in{\rm Spec}(T)$ such that
$H^0_{b,\alpha}(X)$ is non-trivial, and it is natural to fix such $\alpha$ and consider the
Szeg\H{o} kernel for the space $H^0_{b,m\alpha}(X)$ as $m\to+\infty$. This  is the motivation of
our work. In fact, in \cite{HermannHsiaoLi2017embedding}*{Section 3} they pointed out that for an
irregular Sasakian manifold $X$, the $\mathbb{R}$-action on $X$ comes from a CR torus action
denoted by $T^{d+1}$, the Reeb vector field $T$ can also be induced by $T^{d+1}$ and the spectrum
${\rm Spec}(T)=\{\mu_0 p_0+\mu_1 p_1+\cdots+\mu_d p_d:(p_0,p_1,\cdots,p_d)\in\mathbb{Z}^{d+1}\}$,
where $\mu_0,\cdots\mu_d$ are real numbers linearly independent over $\mathbb{Q}$. Hence the
problem above
is equivalent to the study of the Szeg\H{o} kernel for $H^0_{b,mp_0,mp_1,\cdots,mp_d}(X)$ as
$m\to+\infty$ for some lattice point $(p_0,p_1,\cdots,p_d)$. For simplicity, we take $p_0=1$ in
this article. Write $T^{d+1}=S^1\times T^d$. Suppose that the $S^1$-action is CR and
transversal and the $T^d$-action is CR. We prove that if the lattice point
$(-p_1,\cdots,-p_d)$ is a regular value of the associate CR moment map $\mu$, see (\ref{CR moment map}), then there is a
full asymptotic expansion of the torus equivariant Szeg\H{o} kernel as $m\to +\infty$ under certain
assumptions of the positivity of Levi form and the torus action on $Y:=\mu^{-1}(-p_1,\cdots,-p_d)$. In particular, for $\alpha:=\mu_0+\mu_1 p_1+\cdots+\mu_d p_d\in{\rm Spec}(T)$, the space
$H^0_{b,m\alpha}(X)$ is non-trivial as $m\to+\infty$. 

We now briefly introduce some notations and the main results. Let $T_0$ be the induced vector field of the $S^1$-action, and $T_1,\cdots,
T_d$ be the ones for $T^d$-action. In other words, $T_0u(x):=\left.\frac{\partial}{\partial\theta}\right|_{\theta=0}u(e^{i\theta}\circ x)$
and 
$T_ju(x):=\left.\frac{\partial}{\partial\theta_j}\middle|\right._{\theta_j=0}u\left((1,\cdots ,e^{i\theta_j},\cdots ,1)\circ x\right)$, for $j=1,\cdots,d$ and $u\in\mathscr{C}^\infty(X)$. We define $T_0$ and $T_j$ act on $(0,q)$-forms via Lie derivatives $\mathcal{L}_{T_0}$ and $\mathcal{L}_{T_j}$, $j=1,\cdots,d$, respectively. Choose $T_0$ to be the Reeb vector field on $X$, see (\ref{Reeb vector field}), and $\omega_0$
to be its dual one form, see (\ref{omega_0}). Fix a lattice point $(p_1,\cdots,p_d)\in\mathbb{Z}^d$ and any
$m\in\mathbb{N}:=\{1,2,3,\cdots\}$, we consider the space of equivariant smooth $(0,q)$-forms
$$
\Omega^{(0,q)}_{m,mp_1,\cdots,mp_d}(X):=\{u\in\Omega^{(0,q)}(X):-iT_0u=mu,-iT_ju=mp_j
u,j=1,\cdots,d\}.
$$
Since the group action is assumed to be CR, we can take the $\overline{\partial}_b$-subcomplex
$\left(\overline{\partial}_b, \Omega^{(0,\bullet)}_{m,mp_1,\cdots,mp_d}(X)\right)$ and define the
corresponding Kohn--Rossi cohomology $H^q_{b,m,mp_1,\cdots,mp_d}(X)$. Here $\overline{\partial}_b$ we mean the tangential Cauchy--Riemann operator on $X$ with respect to $T_0$, see (\ref{tangential Cauchy Riemann operator}). Let $\langle\cdot|\cdot\rangle$ be a torus invariant Hermitian metric on
$\mathbb{C}TX$ and let $(\cdot|\cdot)$ be the torus invariant $L^2$-inner product on
$\Omega^{(0,q)}(X)$ induced by $\langle\cdot|\cdot\rangle$. Let $L^2_{(0,q),m,mp_1,\cdots,mp_d}(X)$ be the completion of
$\Omega^{(0,q)}_{m,mp_1,\cdots,mp_d}(X)$ with respect to the given torus invariant $L^2$-inner product $(\cdot|\cdot)$ and take the Gaffney extension of Kohn Laplacian $\Box^{(q)}_b$ to the $L^2$-space, see (\ref{Gaffney extension}), then we have the Hodge theorem such that
$$
H^q_{b,m,mp_1,\cdots,mp_d}(X)\cong\mathcal{H}^q_{b,m,mp_1,\cdots,mp_d}(X):=L^2_{(0,q),m,mp_1,\cdots,
mp_d}(X)\cap\ker\Box^{(q)}_b\subset\Omega^{(0,q)}_{m,mp_1,\cdots,mp_d}(X),
$$
and that $H^q_{b,m,mp_1,\cdots,mp_d}(X)$ is finite dimensional for each $q=0,\cdots,n$, though the
Kohn Laplacian $\Box^{(q)}_b$ is not elliptic, and in general it may not be hypoelliptic,
neither. We use the notation
$\Pi^{(q)}_{m,mp_1,\cdots,mp_d}(x,y)$ for the torus equivariant Szeg\H{o} kernel, which is the
distribution kernel of the orthogonal projection
$\Pi^{(q)}_{m,mp_1,\cdots,mp_d}:L^2_{(0,q)}(X)\to\mathcal{H}^q_{b,m,mp_1,\cdots,mp_d}(X)$ with respect to $(\cdot|\cdot)$.

Because the group action here is CR, we can check that the one form $\omega_0$ is torus invariant. We hence consider the torus invariant CR moment
map
\begin{equation}
    \label{CR moment map}
    \mu:X\to\mathbb{R}^d,~\mu(x):=\bigg(\langle\omega_0(x),T_1(x)\rangle,\cdots,\langle\omega_0(x),T_d(x)\rangle\bigg),
\end{equation}
where we identify $(T_e T^d)^*\cong(\mathbb{R}^d)^*\cong\mathbb{R}^d$. In this work, we need
\begin{assumption}
\label{assumption 1}
The given lattice point
$(-p_1,\cdots,-p_d)\in\mathbb{Z}^d$ is a regular value of $\mu$.
\end{assumption}
\begin{assumption}
\label{assumption 2}
The torus action $S^1\times T^d$ is free near $Y$.
\end{assumption}
\begin{assumption}
\label{assumption 3}
The induced Levi form $\mathcal{L}$ is positive near the set $Y:=\mu^{-1}(-p_1,\cdots,-p_d)$
\end{assumption}
The main result in this work is as follows:
\begin{theorem}
\label{Major theorem 1}
Let $(X,T^{1,0}X)$ be a compact connected CR manifold with $2n+1$, $n\geq 2$, admitting a
$S^1\times T^d$ action, where the $S^1$-part is CR and transversal and the $T^d$-part is only required to be CR. For the lattice point $(p_1,\cdots,p_d)\in\mathbb{Z}^d$ and $T^d$-invariant CR moment map $\mu$ satisfying Assumptions \ref{assumption 1}, \ref{assumption 2} and \ref{assumption 3}, then we have the following full asymptotic expansion for the torus equivariant Szeg\H{o} kernel: On one hand, let $\Omega$ be an open set containing $Y$, then 
$$
\Pi^{(q)}_{m,mp_1,\cdots,mp_d}(x,y)=O(m^{-\infty})
$$
on $(X\setminus\Omega)\times(X\setminus\Omega)$ if $q\in\{0,\cdots,n\}$. On the other hand, for each $p\in Y$, we can find a neighborhood denoted by $D_p$, such that 
$$
\Pi^{(q)}_{m,mp_1,\cdots,mp_d}(x,y)=O(m^{-\infty})
$$
on $D_p\times D_p$ if $q\in\{1,\cdots,n\}$. Finally, on $D_p\times D_p$,
$$
\Pi^{(0)}_{m,mp_1,\cdots,mp_d}(x,y)\equiv e^{im f(x,y)}b(x,y,m)~{\rm mod}~O(m^{-\infty}).
$$
Here, the phase function $f\in\mathscr{C}^\infty(D_p\times D_p)$ satisfies ${\rm Im}f\geq 0$, $f(x,x)=0$ for all $x\in Y\cap D_p$ and $d_x f(x,x)=-\omega_0(x)$, $d_y f(x,x)=\omega_0(x)$ for all $x\in Y\cap D_p$;
also, the symbol satisfies
\[
\begin{split}
\quad
&b(x,y,m)\in S^{n-\frac{d}{2}}_{{\rm loc}}(1;D_p\times D_p)\\
&b(x,y,m)\sim\sum_{j=0}^\infty m^{n-\frac{d}{2}-j}b_j(x,y)~\text{in}~S^{n-\frac{d}{2}}_{{\rm loc}}(1;D_p\times D_p)
\end{split}
\]
where $b_j(x,y)\in\mathscr{C}^\infty(D_p\times D_p)$, $j=0,1,2,\cdots$ and $b_0(x,x)>0$ for all $x\in Y\cap D_p$.
\end{theorem}
Here, for continuous functions $f$ and $g$ on an open set $W\subset\mathbb{R}^N$, we use the notation $f\approx g$ on $W$ if  there is a constant $C>0$ such that $\frac{1}{C}g\leq f\leq C g$ on $W$; we refer the semi-classical notations such as
$O(m^{-\infty})$, $A=B+O(m^{-\infty})$, $S^k_{\rm loc}(1;D_p\times D_p)$, and asymptotic sums $\sim$ in $S^k_{\rm loc}$ to Section 2.3.

 Note that Theorem \ref{Major theorem 1} holds on a class of manifolds slightly more general than Sasakian manifolds, for we only assume the Levi form is positive near the submanifold $Y$ instead of being positive on the whole $X$. Also, from Theorem \ref{Major theorem 1}, we can conclude:
 \begin{corollary}
 Let $(X,T^{1,0}X)$ be an irregular Sasakian manifold and $T$ be the fundamental vector field induced by the prescribed $\mathbb{R}$-action on $X$. It is known that $\mathbb{R}$-action comes from a $S^1\times T^d$-action. Assume $T=\mu_0T_0+\mu_1 T_1+\cdots+\mu_d T_d$, where $T_0$ is the vector field induced by $S^1$-action and $T_1,\cdots,T_d$ are the vector fields induced by $T^d$-action and $\mu_0,\mu_1,\cdots,\mu_d$ are real numbers linearly independent over $\mathbb{Q}$. If the lattice point $(p_1,\cdots,p_d)\in\mathbb{Z}^d$ satisfies Assumptions \ref{assumption 1},\ref{assumption 2}, for $\alpha:=\mu_0+\mu_1 p_1+\cdots+\mu_d p_d\in {\rm Spec}(T)$, $H^0_{b,m\alpha}(X)$ is non-trivial as $m\to\infty$. Here, $H^0_{b,m\alpha}(X):=\{u\in\mathscr{C}^\infty(X):\overline{\partial}_b u=0,~Tu=im\alpha u\}$.
 \end{corollary}
 We now illustrate the strategy for the proof of the theorem. Since $X$ is NOT assumed to be strongly pseudoconvex in this paper, to establish the asymptotics of torus equivariant Szeg\H{o} kernel, we do not study Szeg\H{o} projector $\Pi^{(q)}_{m,mp_1,\cdots,mp_d}$ directly; instead. we need to consider a number $\lambda>0$ and examine the Szeg\H{o} projector for lower energy forms $\Pi^{(q)}_{\leq\lambda,m,mp_1,\cdots,mp_d}$, which is the orthogonal projection from $L^2_{(0,q)}(X)$ to $L^2_{(0,q),m,mp_1,\cdots,mp_d}(X)\cap E((-\infty,\lambda])$. Here, $E((-\infty,\lambda])$ is the spectral projector and $E$ is the spectral measure for the self-adjoint operator $\Box_b^{(q)}$ under Gaffney extension \ref{Gaffney extension}, respectively. We need to  establish the content in Theorem \ref{Major theorem 1} in the following version:
\begin{theorem}[=Theorem \ref{Main theorem 2}+\ref{Main theorem 3}]
\label{Main theorem 1}
With the same notations and assumptions used in Theorem \ref{Major theorem 1} and any fixed $\lambda>0$, on one hand, for any open set $\Omega$ containing $Y$, 
$$
\Pi^{(q)}_{\leq\lambda,m,mp_1,\cdots,mp_d}(x,y)=O(m^{-\infty})
$$
on $(X\setminus\Omega)\times(X\setminus\Omega)$ if $q\in\{0,\cdots,n\}$. On the other hand, for each $p\in Y$, we can find a neighborhood denoted by $D_p$, such that 
$$
\Pi^{(q)}_{\leq\lambda,m,mp_1,\cdots,mp_d}(x,y)=O(m^{-\infty})
$$
on $D_p\times D_p$ if $q\in\{1,\cdots,n\}$. Finally, on $D_p\times D_p$,
$$
\Pi^{(0)}_{\leq\lambda,m,mp_1,\cdots,mp_d}(x,y)\equiv e^{im f(x,y)}b(x,y,m)~{\rm mod}~O(m^{-\infty})
$$
for the same phase function $f(x,y)\in\mathscr{C}^\infty(D_p\times D_p)$ and symbol $b(x,y,m)\in S^{n-\frac{d}{2}}_{\rm loc}(1;D_p\times D_p)$ in Theorem \ref{Major theorem 1}.
\end{theorem}
We will see in the beginning of the Section 3 that combining the spectral property for Kohn Laplacian and Theorem \ref{Main theorem 1} for the case $q=1$, there is:
\begin{theorem}[=Theorem \ref{Main theorem 4}]
\label{Major theorem 3}
With the same notations and assumptions used in Theorem \ref{Major theorem 1}, then for any $\lambda>0$, as $m\to +\infty$,
$$
\Pi^{(0)}_{\leq\lambda,m,mp_1,\cdots,mp_d}=\Pi^{(0)}_{m,mp_1,\cdots,mp_d}.
$$
\end{theorem}
Thus, from Theorem \ref{Major theorem 3}, Theorem \ref{Main theorem 1} actually implies Theorem \ref{Major theorem 1}.

\section{Set up and notation}
In this section, we recall some basic language in CR geometry, definition and properties of torus equivariant Szeg\H{o} kernel, and tools in semi-classical analysis and microlocal analysis.
\subsection{Cauchy--Riemann manifold and Kohn Laplacian} We follow the presentation in \cite{BG1988}*{Cahpter 4} and \cite{Hsiao2010}*{Cahpter 2}. Let $X$ be a smooth orientable manifold of real dimension $2n+1,~n\geq 1$, we say $X$ is a Cauchy--Riemann manifold (CR manifold for short) if there is a subbundle
$T^{1,0}X\subset\mathbb{C}TX$, such that
\begin{enumerate}
\item $\dim_{\mathbb{C}}T^{1,0}_{p}X=n$ for any $p\in X$.
\item $T^{1,0}_p X\cap T^{0,1}_p X=\{0\}$ for any $p\in X$, where $T^{0,1}_p X:=\overline{T^{1,0}_p X}$.
\item For $V_1, V_2\in \mathscr{C}^{\infty}(X,T^{1,0}X)$, then $[V_1,V_2]\in\mathscr{C}^{\infty}(X,T^{1,0}X)$, where 
$[\cdot,\cdot]$ stands for the Lie bracket between vector fields. 
\end{enumerate}
For such subbundle $T^{1,0}X$, we call it a CR structure of the CR manifold $X$. Fix a Hermitian metric $\langle\cdot|\cdot\rangle$ on $\mathbb{C}TX$ such that $T^{1,0}X\perp T^{0,1}X$. For dimension reason and the assumption that $X$ is orientable, we can always take a non-vanishing real global vector field $T$ (Reeb vector field) such that for all $x\in X$, we have the orthogonal decomposition
\begin{equation}
    \label{Reeb vector field}
    T^{1,0}_x X\oplus T^{0,1}_x X\oplus\mathbb{C}T(x)=\mathbb{C}T_x X
\end{equation}
and $\langle T|T\rangle=1$ on $X$. Denote $\langle\cdot,\cdot\rangle$ to be the paring by duality between vector fields and differential forms, and let $\Gamma:\mathbb{C}T_x X\to \mathbb{C}T_x^*X$ be the anti-linear map given by $\langle u|v\rangle=\langle u,\Gamma v\rangle$ for $u,v\in\mathbb{C}T_x X$, then we can take the induced Hermitian metric on $\mathbb{C}T^*X$ by $\langle u|v\rangle:=\langle\Gamma^{-1}v|\Gamma^{-1}u\rangle$ for $u,v\in\mathbb{C}T^*_x X$. Put 
$$
T^{*{1,0}}X:=\Gamma(T^{1,0}X)={(T^{0,1}X\oplus\mathbb{C}T)}^{\perp}\subset\mathbb{C}T^*X,~T^{*{0,1}}X:=\overline{T^{*{1,0}}X}
$$
and 
\begin{equation}
    \label{omega_0}
    \omega_0:=-\Gamma(T),
\end{equation}
which is a globally defined non-vanishing 1-form satisfying
$$
T^{*1,0}_x X\oplus T^{*0,1}_x X\oplus\mathbb{C}\omega_0(x)=\mathbb{C}T^*_xX,~\langle\omega_0,T^{1,0}X\oplus T^{0,1}X\rangle=0
\text{ and }\langle\omega_0,T\rangle=-1.
$$
We define the Levi form, which is a globally defined $(1,1)$-form, by
$$
\mathcal{L}_x({u},\overline{v}):=\frac{1}{2i}\left\langle\omega_0(x),\left[\mathring{u},\overline{\mathring{v}}\right](x)\right\rangle,
$$
where $\mathring{u},\mathring{v}\in\mathscr{C}^\infty(X,T^{1,0} X)$ such that $\mathring{u}(x)=u\in T_x^{1,0}X$ and $\mathring{v}(x)=v\in T_x^{1,0}X$. Note that by Cartan's formula we can also express the Levi form by 
$$
\mathcal{L}_x(u,\overline{v})=\frac{-1}{2i}\left\langle d\omega_0(x),u\wedge\overline{v}\right\rangle,~u,v\in T^{1,0}_x X.
$$
In other words,
$$
\mathcal{L}_x:=\left.\frac{-1}{2i}d\omega_0(x)\right|_{T_x^{1,0}X}.
$$
Take the Hermitian metric on $\Lambda^r\mathbb{C}T^*X$ by
$$
\langle u_1\wedge\cdots \wedge u_r|v_1\wedge\cdots v_r\rangle=\det\left(\bigg(\langle u_j|u_k\rangle\bigg)_{j,k=1}^r\right),~\text{where}~u_j,v_k\in\mathbb{C}T^*X~,j,k=1,\cdots,r,
$$
and the orthogonal projection 
$$
\pi^{(0,q)}:\Lambda^q\mathbb{C}T^*X\to T^{*0,q}X:=\Lambda^q(T^{*0,1}X)
$$
with respect to this Hermitian metric. The tangential Cauchy--Riemann operator is defined to be 
\begin{equation}
    \label{tangential Cauchy Riemann operator}
    \overline{\partial}_b:=\pi^{(0,q+1)}\circ d:\mathscr{C}^{\infty}(X,T^{*0,q}X)\to\mathscr{C}^{\infty}(X,T^{*0,q+1}X).
\end{equation}
By Cartan's formula, we can check that 
$$
\overline{\partial}_b^2=0.
$$
Take the $L^2$-inner product $(\cdot|\cdot)$ on $\mathscr{C}^\infty(X,T^{*0,q}X)$ induced by $\langle\cdot|\cdot\rangle$ via
$$
(f|g):=\int_X\langle f|g\rangle dV_X,~f,g\in\mathscr{C}^\infty(X,T^{*0,q}X),
$$
where $dV_X$ is the volume form with expression 
$$
dV_X(x)=\sqrt{\det\left(\left\langle\frac{\partial}{\partial x_j}\middle|\frac{\partial}{\partial x_k}\right\rangle\right)_{j,k=1}^n}dx_1\wedge\cdots \wedge dx_{2n+1}
$$
in local coordinates $(x_1,\cdots,x_{2n+1})$, and we write $\overline{\partial}_b^*$ to denote the formal adjoint of $\overline{\partial}_b$ with respect to the $L^2$-inner product $(\cdot|\cdot)$. Denote $\Omega^{(0,q)}(X):=\mathscr{C}^{\infty}(X,T^{*0,q}X)$, then the Kohn Laplacian is the operator
$$
\Box^{(q)}_b:=\overline{\partial}_b^*\overline{\partial}_b+\overline{\partial}_b\overline{\partial}_b^*:\Omega^{(0,q)}(X)\to\Omega^{(0,q)}(X).
$$
We have the following Bochner--Kodaira--Kohn formula for $\Box^{(q)}_b$, see \cite{Hsiao2010}*{Proposition 2.3}:
\begin{theorem}
\label{Bochner formula}
Let $p\in X$, $\{e_j(x)\}_{j=1}^n$ be an orthonormal frame of $T^{*0,1}_x X$ varying smoothly with $x$ in a neighborhood of $p$, and $\{L_j(x)\}_{j=1}^n$ be the dual frame of $T_x^{0,1}X$. Then we have
$$
\Box_b^{(q)}=\sum_{j=1}^n L_j^*L_j+\sum_{j,k=1}^n(e_j\wedge e_k^{\wedge,*})[L_j,L_k^*]+\epsilon(L)+\epsilon(\overline{L})+\text{zero order terms}.
$$
Here, $\epsilon(L):=\sum_{j=1}^n a_j L_j$, $\epsilon(\overline{L}):=\sum_{j=1}^n b_j \overline{L}_j$ with smooth coefficeints $a_j$ and $b_j$ for $j,k=1,\cdots,n$. Also, for each $j,k=1,\cdots,n$, $L_j^*$ is the formal adjoint of the differential operator $L_j$, and $e_k^{\wedge,*}$ is the adjoint of $e_k\wedge$ given by $\langle e_k\wedge u| v\rangle=\langle u|e_k^{\wedge,*} v\rangle$ for $u\in T^{*0,q}X$ and $v\in T^{*0,q+1}X$.
\end{theorem}
\subsection{Torus equivariant Szeg\H{o} kernel}
From now on, we assume that $X$ admits a torus action in the form of $S^1\times T^d$. Consider the vector fields 
\begin{equation*}
    T_0u(x):=\left.\frac{\partial}{\partial\theta}\right|_{\theta=0}u(e^{i\theta}\circ x)
\end{equation*}
and
\begin{equation*}
    T_ju(x):=\left.\frac{\partial}{\partial\theta_j}\middle|\right._{\theta_j=0}u\left((1,\cdots ,e^{i\theta_j},\cdots ,1)\circ x\right),
\end{equation*}
where $u\in\mathscr{C}^{\infty}(X),~x\in X$, $j=1,\cdots,d$. Note that $T_0$ and
$T_1,\cdots T_d$ are the induced vector fields of the circle action and torus
action, respectively. We also assume that the circle action here is CR and
transversal, i.e.~$T_0$ satisfies
 \begin{equation}
     \label{T_0 is CR}
      \left[T_0,\mathscr{C}^{\infty}(X,T^{1,0}X)\right]\subset\mathscr{C}^{\infty}(X,T^{1,0} X)
 \end{equation}
 and
 \begin{equation}
     \label{T_0 is transversal}
       \mathbb{C}T_0(x)\oplus T^{1,0}_x X\oplus T^{0,1}_x X=\mathbb{C}T_x X\text{ for all } x\in X;
 \end{equation}
the torus action here is also assumed to be CR, i.e.~for all $j=1,\cdots,d$, $T_j$ has the property
 \begin{equation}
     \label{T_j are CR}
      \left[T_j,\mathscr{C}^{\infty}(X,T^{1,0}X)\right]\subset\mathscr{C}^{\infty}(X,T^{1,0} X).
 \end{equation}
We choose $T_0$ to be our Reeb vector field of $X$, i.e.~$T:=T_0$. Accordingly, the cooresponding dual one form $\omega_0$ is torus invariant, because $(e^{i\theta_1},\cdots,e^{i\theta_d})\circ e^{i\theta}=e^{i\theta}\circ (e^{i\theta_1},\cdots,e^{i\theta_d})$. Benefit from the CR and transversal circle action, we have the BRT coordinates patch from Baouendi--Rothschild--Tr{\`e}ves \cite{BRT1985}*{Proposition I.1}:
\begin{theorem}
\label{BRT patch}
Assume $X$ is a CR manifold admitting a CR and transversal circle action. Fix a point
$p\in X$, then there exists $\epsilon>0$, $\delta>0$ and an open neighborhood $D:=\{ (z,\theta):|z|<\epsilon,|\theta|<\delta\}$ and local coordinates $(x_1,x_2,\cdots ,x_{2n-1},x_{2n},x_{2n+1})=(z_1,\cdots ,z_n,\theta)$ near $p$, where
$z_j:=x_{2j-1}+ix_{2j}, ~j=1,\cdots,n,~\theta:=x_{2n+1}$,
such that $(z(p),\theta(p))=(0,0)$ and the fundamental vector field induced by circle
action is $T_0=\frac{\partial}{\partial\theta}$. Also, we can find a real valued function
$\phi(z)=\sum_{j=1}^n \lambda_j|z_j|^2+O(|z|^3)\in\mathscr{C}^{\infty}(D,\mathbb{R})$, where $\{\lambda_j\}_{j=1}^n$ are eigenvalues of the Levi form of $X$ at $p$, such that
$\left\{Z_j:=\frac{\partial}{\partial z_j}+i\frac{\partial\phi(z)}{\partial z_j}\frac{\partial}{\partial\theta}\right\}_{j=1}^n$ forms a basis of $T^{1,0}_x X$ $\text{for all }x\in D$. Moreover, we can take $\delta=\pi$ when the action at the point
$p$ is free.
\end{theorem}
By (\ref{T_0 is CR}) and (\ref{T_j are CR}), we can check that $\overline{\partial}_b T_j=T_j\overline{\partial}_b$ for all $j=0,1,\cdots,d$. For any given lattice point $(p_1,\cdots,p_d)\in\mathbb{Z}^d$, we put
$$
\Omega^{(0,q)}_{m,mp_1,\cdots,mp_d}(X):=\{u\in\Omega^{(0,q)}(X):-iT_0u=mu,-iT_ju=mp_ju,~j=1,\cdots,d\}.
$$
For $\overline{\partial}_b:\Omega^{(0,q)}_{m,mp_1,\cdots,mp_d}(X)\to \Omega^{(0,q+1)}_{m,mp_1,\cdots,mp_d}(X)$, we can take the $\overline{\partial}_b$-subcomplex $\left(\overline{\partial}_b,\Omega^{(0,q)}_{m,mp_1,\cdots,mp_d}(X)\right)$
and the corresponding Kohn--Rossi cohomology
$$
H^q_{b,m,mp_1,\cdots,mp_d}(X):=\frac{\ker\overline{\partial}_b:\Omega^{(0,q)}_{m,mp_1,\cdots,mp_d}(X)\to\Omega^{(0,q+1)}_{m,mp_1,\cdots,mp_d}(X)}{{\rm Im}\overline{\partial}_b:\Omega^{(0,q-1)}_{m,mp_1,\cdots,mp_d}(X)\to\Omega^{(0,q)}_{m,mp_1,\cdots,mp_d}(X)}
$$
From now on, we pick a $S^1\times T^d$-invariant Hermitian metric $\langle\cdot|\cdot\rangle$ on $\mathbb{C}TX$. Take the restriction $\overline{\partial}_{b,m,mp_1,\cdots,mp_d}:=\left.\overline{\partial}_b\right.|_{\Omega^{(0,q)}_{m,mp_1,\cdots,mp_d}(X)}$, then we can check that the formal adjoint of $\overline{\partial}_{b,m,mp_1,\cdots,mp_d}$ satisfies
$$
\overline{\partial}_{b,m,mp_1,\cdots,mp_d}^*=\left.\overline{\partial}_b^*\right.|_{\Omega^{(0,q+1)}_{m,mp_1,\cdots,mp_d}(X)}:{\Omega^{(0,q+1)}_{m,mp_1,\cdots,mp_d}(X)}\to {\Omega^{(0,q)}_{m,mp_1,\cdots,mp_d}(X)}.
$$
So we can consider the Fourier component of Kohn Laplacian by
$$
\Box^{(q)}_{b,m,mp_1,\cdots,mp_d}:=\left.\Box^{(q)}_b\right|_{\Omega^{(0,q)}_{m,mp_1,\cdots,mp_d}(X)}:\Omega^{(0,q)}_{m,mp_1,\cdots,mp_d}(X)\to\Omega^{(0,q)}_{m,mp_1,\cdots,mp_d}(X).
$$

We pause here for a while to handle some issue on extending Kohn Laplacian to $L^2$-space as a self-adjoint operator. Let $L^2_{(0,q)}(X)$ to be the completion of $\Omega^{(0,q)}(X)$ with respect to the torus invariant $L^2$-inner product $(\cdot|\cdot)$ induced by $\langle\cdot|\cdot\rangle$. Denote $\|u\|_X^2:=(u|u)$. Define the weak maximal extension of $\Box_b^{(q)}$ by
\begin{equation}
    \label{weak maximal extension}
    \begin{split}
            ~
            &{\rm Dom}(\Box^{(q)}_{b,{\rm max}}):=\{u\in L^2_{(0,q)}(X):\Box^{(q)}_b u\in L^2_{(0,q)}(X)~\text{in the distribution sense}\},\\
            & \Box_{b,{\rm max}}^{(q)}u=\Box_b^{(q)}u~\text{in distribution sense, for all}~u\in{\rm Dom}(\Box_{b,{\rm max}}^{(q)})
    \end{split}
\end{equation}
as in \cite{MaMarinescu2007}*{Section 3.1}. For such extension, $\Box_b^{(q)}$ may not be a self-adjoint operator, because it is non-elliptic, and it could also be non-hypoelliptic. So in general we have to consider the Gaffney extension as in \cite{MaMarinescu2007}*{Proposition 3.1.2}, and this extension can make $\Box_b^{(q)}$ to be self-adoint. Precisely, take the maximal extension
$\overline{\partial}_b:=\overline{\partial}_{b,{\rm max}}$ and the Hilbert adjoint of $\overline{\partial}_b$ on the $L^2$-space by
$$
\overline{\partial}_{b,H}^*:{\rm Dom}(\overline{\partial}_{b,H}^*)\subset L^2_{(0,q+1)}(X)\to  L^2_{(0,q)}(X),
$$
where the domain is given by
$$
{\rm Dom}(\overline{\partial}_{b,H}^*):=\{v\in L^2_{(0,q)}(X):\text{for all}~u\in {\rm Dom}(\overline{\partial}_b),~\text{the operator}~u\mapsto (\overline{\partial}_b u|v)~\text{is bounded linear}\}.
$$
By Riesz representation theorem, for all $v\in {\rm Dom}(\overline{\partial}^*_{b,H})$, there is a $w\in L^2_{(0,q)}(X)$ such that $(\overline{\partial}_b u|v)=(u|w)$ for all $u\in{\rm Dom}(\overline{\partial}_b)$, and $\overline{\partial}_{b,H}^*v:=w$. Then the Gaffney extension is given by
\begin{equation}
    \label{Gaffney extension}
    \begin{split}
        &~{\rm Dom}(\Box_b^{(q)})
        :=\{u\in{\rm Dom}(\overline{\partial}_b)\cap{\rm Dom}(\overline{\partial}^*_{b,H}):\overline{\partial}_b u\in{\rm Dom}\overline{\partial}^*_{b,H},~\overline{\partial}_{b,H}^*u\in{\rm Dom}\overline{\partial}_b\},\\
        &\Box_b^{(q)}u=(\overline{\partial}^*_{b,H}\overline{\partial}_b+\overline{\partial}_b\overline{\partial}^*_{b,H})u,~\text{for all}~u\in{\rm Dom}(\Box_b^{(q)}).
    \end{split}
\end{equation}
Let $\Omega^{(0,q)}_m(X):=\{u\in \Omega^{(0,q)}(X):-iT_0 u=mu\}$ and $L^2_{(0,q),m}(X)$ be the completion of $\Omega^{(0,q)}_m(X)$ with respect to $(\cdot|\cdot)$. We need the following:
\begin{proposition}
The Gaffney extension and the weak maximal extension for $\Box_b^{(q)}$ coincides on $L^2_{(0,q),m}(X)$, and hence on $L^2_{(0,q),m,mp_1,\cdots,mp_d}(X)$. 
\end{proposition}
\begin{proof}
On one hand, by $\overline{\partial}_{b,H}^*|_{\Omega^{(0,q)}(X)}=\overline{\partial}_{b}^*$, for all $v\in\Omega^{(0,q)}(X)$, 
$$
(\Box_b^{(q)}u|v)=(u|\Box_b^{(q)}v)=(u|\Box^{(q)}_{b,{\rm max}}v)=:(\Box^{(q)}_{b,{\rm max}}u|v),~\text{for all}~u\in{\rm Dom}(\Box_b^{(q)}),
$$
so $\Box_{b,{\rm max}}^{(q)}u=\Box^{(q)}_b u\in L^2_{(0,q)}(X)$ for $u\in {\rm Dom}(\Box_b^{(q)})$. This implies that ${\rm Dom}(\Box_b^{(q)})\subset{\rm Dom}(\Box_{b,max}^{(q)})$. On the other hand, though Theorem \ref{Bochner formula} suggests that $\Box^{(q)}_b$ is not an elliptic operator (since the
principal symbol $\sigma_{\Box_b^{(q)}}(x,\lambda\omega_0(x))=0$, for all $\lambda\in\mathbb{R}\setminus\{0\}$), the operator $\Box^{(q)}_b-T_0^2$ is elliptic. After
applying the elliptic regularity for $\Box^{(q)}_b-T_0^2$ on the space
$L^2_{(0,q),m}(X)$, we can check that 
$$
{\rm Dom}(\Box^{(q)}_{b,{\rm max}})\cap L^2_{(0,q),m}(X)=H^2_{(0,q),m}(X).
$$
Here, we let $H^s_{(0,q)}(X)$ to be the Sobolev space of order $s$ for $(0,q)$ forms on $X$ with respect to a Sobolev norm $\|\cdot\|_s$ induced by the invariant $L^2$-inner product $(\cdot|\cdot)$, and
$H^s_{(0,q),m}(X):=H^s_{(0,q)}(X)\cap L^2_{(0,q),m}(X)$. Accordingly,
${\rm Dom}(\Box_{b,{\rm max}}^{(q)})\cap L^2_{(0,q),m}(X)=H^2_{(0,q),m}(X)\subset{\rm Dom}(\overline{\partial}_b)$; also, by Friedrichs lemma
\cite{MaMarinescu2007}*{Lemma 3.1.3}, for $v\in H^2_{(0,q),m}(X)$, we can find a sequence $\{v_j\}_{j=1}^\infty\subset\Omega^{(0,q)}(X)$ such that $v_j\to v$ in $L^2_{(0,q)}(X)$ and $\overline{\partial}_b^* v_j\to \overline{\partial}_b^* v$ in $L^2_{(0,q-1)}(X)$ as $j\to\infty$. So for all $u\in {\rm Dom}(\overline{\partial}_b)$, there is a constant $C_1:=\|\overline{\partial}_b^* v\|_X<\infty$ such that
$$
|(\overline{\partial}_b u|v)|=\lim_{j\to\infty}|(\overline{\partial}_b u|v_j)|=\lim_{j\to\infty}|(u|\overline{\partial}_b^*v_j)|\leq \lim_{j\to\infty}\|u\|_X\|\overline{\partial}_b^*v_j\|_X=\|\overline{\partial}_b^*v\|_X\|u\|_X\leq C_1 \|u\|_X.
$$
Thus,
$$
{\rm Dom}(\Box_{b,{\rm max}}^{(q)})\cap L^2_{(0,q),m}(X)=H^2_{(0,q),m}(X)\subset{\rm Dom}(\overline{\partial}_b)\cap {\rm Dom}(\overline{\partial}^*_{b,H}).
$$
Next, we check that if $v\in{\rm Dom}(\Box_{b,{\rm max}}^{(q)})\cap L^2_{(0,q),m}(X)=H^2_{(0,q),m}(X)$, then there is a constant $C_2>0$ such that
$$
|(\overline{\partial}_b u|\overline{\partial}_b v)|\leq C_2\|u\|_X~\text{for all}~u\in{\rm Dom}(\overline{\partial}_b).
$$
If this is true, then $\overline{\partial}_b:{\rm Dom}({\Box_{b,{\rm max}}^{(q)}})\to{\rm Dom}(\overline{\partial}^*_{b,H})$. Now, let $w:=\overline{\partial}_b v\in H^1_{(0,q+1),m}(X)$, by Friedrichs lemma again, we can take a sequence $\{w_j\}_{j=1}^\infty\subset\Omega^{(0,q+1)}(X)$ such that  $w_j \to w$ in $L^2_{(0,q)}(X)$ and $\overline{\partial}_b^* w_j\to \overline{\partial}_b^* w$ in $L^2_{(0,q-1)}(X)$. Then for a constant $C_2:=\|\overline{\partial}_b^*w\|_X<\infty$,
$$
|(\overline{\partial}_b u|\overline{\partial}_b v)|=\lim_{j\to\infty}|(\overline{\partial}_b u|w_j)|=\lim_{j\to\infty}|(u|\overline{\partial}_b^* w_j)|\leq \lim_{j\to\infty}\|u\|_X\|\overline{\partial}^*_b w_j\|_X=\|\overline{\partial}_b^* w\|_X\|u\|_X\leq C_2\|u\|_X.
$$
Similarly, we have $\overline{\partial}^*_{b,H}:{\rm Dom}({\Box_{b,{\rm max}}^{(q)}})\to{\rm Dom}(\overline{\partial}_b)$, so we can conclude
$$
{\rm Dom}(\Box^{(q)}_{b,{\rm max}})\cap L^2_{(0,q),m}(X)\subset{\rm Dom}(\Box_b^{(q)})\cap L^2_{(0,q),m}(X).
$$
\end{proof}
With the proposition, from now on, we take the extension
$$
\Box^{(q)}_{b,m,mp_1,\cdots,mp_d}:{\rm Dom}(\Box^{(q)}_{b,m,mp_1,\cdots,mp_d})\subset L^2_{(0,q),m,mp_1,\cdots,mp_d}(X)\to L^2_{(0,q),m,mp_1,\cdots,mp_d}(X)
$$
where
$$
{\rm Dom}\left(\Box^{(q)}_{b,m,mp_1,\cdots,mp_d}\right):=\left\{u\in L^2_{(0,q)}(X):\Box^{(q)}_{b}u\in L^2_{(0,q)}(X)\right\}\cap L^2_{(0,q),m,mp_1,\cdots,mp_d}(X)
$$
and $\Box^{(q)}_{b,m,mp_1,\cdots,mp_d}u:=\Box_b^{(q)}u$ in the distribution sense for all $u\in {\rm Dom}\left(\Box^{(q)}_{b,m,mp_1,\cdots,mp_d}\right)$, to extend $\Box^{(q)}_{b,m,mp_1,\cdots,mp_d}$. We have some standard spectral properties for $\Box^{(q)}_{b,m,mp_1,\cdots,mp_d}$:
\begin{theorem}
\label{Hodge theorem for Kohn's Laplacian}
For $\Box^{(q)}_{b,m,mp_1,\cdots,mp_d}:{\rm Dom}\left(\Box^{(q)}_{b,m,mp_1,\cdots,mp_d}\right)\to L^2_{(0,q)m,mp_1,\cdots,mp_d}(X)$, we have
\begin{enumerate}
    \item $\Box^{(q)}_{b,m,mp_1,\cdots,mp_d}$ is a non-negative and a self-adjoint operator.
    \item The spectrum ${\rm Spec}(\Box^{(q)}_{b,m,mp_1,\cdots,mp_d})$ consists only of eigenvalues, and it is a countable and discrete subset in $[0,\infty)$.
    \item For each $\mu\in{\rm Spec}(\Box^{{q}}_{b,m,mp_1,\cdots,mp_d})$, the space of eigenforms 
    $$
    \mathcal{H}^q_{b,\mu,m,mp_1,\cdots,mp_d}(X):=\{u\in {\rm Dom}(\Box^{(q)}_{b,m,mp_1,\cdots,mp_d}):\Box^{(q)}_bu=\mu u\}
    $$
    is a finite dimensional subspace of $\Omega^{(0,q)}_{m,mp_1,\cdots,mp_d}(X)$.
    \item 
    ${H}^q_{b,m,mp_1,\cdots ,mp_d}(X)\cong\mathcal{H}^q_{b,m,mp_1,\cdots ,mp_d}(X):=\mathcal{H}^q_{b,\mu=0,m,mp_1,\cdots ,mp_d}(X)$.
\end{enumerate}
\end{theorem}
\begin{proof}
The argument is almost the same as the case \cite{ChengHsiaoTsai2017}*{Section 3} when only circle action  is involved, and for the modification to torus action we refer to the proof in \cite{HermannHsiaoLi2017embedding}*{Section 4}. 
\end{proof}
Let $n_q:=\dim\mathcal{H}^q_{b,m,mp_1,\cdots ,mp_d}<\infty$, and the torus equivariant Szeg{\H{o}} kernel function
\begin{equation}
\label{reproducing kernel at diagonal}
        {\rm Tr}\Pi^{(q)}_{m,mp_1,\cdots ,mp_d}(x)
        :=\sum_{j=1}^{n_q}|f_{j}^q(x)|_h^2:=\sum_{j=1}^{n_q}\left\langle f_j^q(x)\middle|f_j^q(x)\right\rangle
\end{equation}
where $\{f_j^q\}_{j=1}^{n_q}$ is an orthonormal basis for $\mathcal{H}^q_{b,m,mp_1,\cdots ,mp_d}$.
For an open set $D\subset X$, take $\{e_j(x)\}_{j=1}^n$ varying smoothly for $x\in D$ such that 
$\{e_j(x)\}_{j=1}^n$ is an orthonormal basis for $T^{0,1}_x X$ at every $x\in D$. For a strictly increasing index set $J=\{j_1,\cdots,j_q\}$ with $|J|=q$,  write $e^J:=e_{j_1}\wedge\cdots\wedge e_{j_q}$ and for any
$u\in\Omega^{(0,q)}(X)$, write $u(x)=\sum_{|J|=q}'u_J(x) e^J(x)$, where
$\Sigma'_{|J|=q}$ we mean the summation only over a strictly increasing index set. Then we can
find the torus equivariant Szeg{\H{o}} kernel function is the peak function
similar in \cite{HsiaoMarinescu2012}*{Lemma 2.1}, i.e.~
        $$
        {\rm Tr}\Pi^{(q)}_{m,mp_1,\cdots ,mp_d}(x)
        ={\sum_{|J|=q}}'\sup\left\{|u_J(x)|_h^2:u\in\mathcal{H}^q_{b,m,mp_1,\cdots,mp_d}(X),\|u\|_X^2=1\right\}.
        $$
        Also, for all $q=0,\cdots,n$, the torus equivariant Szeg\H{o} kernel
        $$
        \Pi^{(q)}_{m,mp_1,\cdots ,mp_d}(x,y)\in\mathscr{C}^\infty\left(X\times X,{\rm Hom}(T^{*0,q}X,T^{*0,q}X)\right)
        $$ 
        is the distribution kernel of the orthogonal projection
        $$
        \Pi^{(q)}_{m,mp_1,\cdots,mp_d}:L^2_{(0,q)}(X)\to\mathcal{H}^{(q)}_{m,mp_1,\cdots,mp_d}(X)
        $$
        with respect to $(\cdot|\cdot)$. By Theorem \ref{Hodge theorem for Kohn's Laplacian}, the projection is a smoothing operator, and we can check that locally on $D$, we have the expression
        \begin{equation}
                \Pi^{(q)}_{m,mp_1,\cdots ,mp_d}(x,y)
                ={\sum_{|I|=q}}'{\sum_{|J|=q}}'\Pi^{(q)}_{m,mp_1,\cdots ,mp_d,I,J}(x,y)e^I(x) \otimes\left(e^J(y)\right)^*
        \end{equation}
        in the sense that
        \begin{equation}
                    \Pi^{(q)}_{m,mp_1,\cdots ,mp_d}u(x)={\sum_{|I|=q}}'{\sum_{|J|=q}}'\left(\int_D\Pi^{(q)}_{m,mp_1,\cdots ,mp_d,I,J}(x,y)u_J(y)dV_X(y)\right)e^I(x)
        \end{equation}
        for $u=\sum_{{|J|=q}}'u_J e^J\in\Omega^{(0,q)}(X)$. We can check that for all strictly increasing index set $I,J$, $|I|=|J|=q$,
        \begin{equation}
        \label{diagonal kernel is kernel function}
             {\sum_{|J|=q}}'\left\langle\Pi^{(q)}_{m,mp_1,\cdots ,mp_d}(x,x) e^J(x)|e^J(x)\right\rangle={\rm Tr}\Pi^{(q)}_{m,mp_1,\cdots ,mp_d}(x),
        \end{equation}
        Here, $\Pi^{(q)}_{m,mp_1,\cdots ,mp_d,I,J}(x,y)\in\mathscr{C}^\infty(D\times D)$ for all strictly increasing index set $I$ and $J$, $|I|=|J|=q$. Moreover, we can also check that for all strictly invrasing index set $I$ and $J$, $|I|=|J|=q$,
        \begin{equation}
        \label{reproducing kernel in x and y}
\Pi^{(q)}_{m,mp_1,\cdots,mp_d,I,J}(x,y)=\sum_{j=1}^{n_q}f^q_{j,I}(x)\overline{f^q_{j,J}}(y),
        \end{equation}
        where $f^q_{j}=\sum_{|K|=q}'f^q_{j,K}e^K,j=1,\cdots,n_q$, is an orthonormal basis for $\mathcal{H}^q_{b,m,mp_1,\cdots,mp_d}(X)$. 
\subsection{Notations in semi-classical analysis and microlocal analysis}
We here present some convention in semi-classical analysis and microlocal analysis \cites{DS1999,GrigisSjostrand1994,MS1975} to describe and calculate the asymptotic behavior of the torus equivariant Szeg\H{o} kernel. Let $U$ be an open set in $\mathbb{R}^{n_1}$ and let $V$ be an open set in $\mathbb{R}^{n_2}$. Let $E$ and $F$ be vector bundles over $U$ and $V$, respectively. Let $\mathscr{C}^\infty_0(V,F)$ and $\mathscr{C}^\infty(U,E)$ be the space of smooth sections of $F$ over $V$ with compact support in $V$ and the space of smooth sections of $E$ over $U$, respectively; $\mathscr{D}'(U,E)$ and $\mathscr{E}'(V,F)$ be the space of distributional sections of $E$ over $U$ and the space of distributional sections of $F$ over $V$ with compact support in $V$, respectively. We say an $m$-dependent continuous linear operator
$$
A_m:\mathscr{C}^\infty_0(V,F)\to\mathscr{D}'(U,E) 
$$
is $m$-negligible if
\begin{enumerate}
    \item for all $m$ large enough $A_m$ is a smoothing operator, which is equivalent to (\cite{hormander2003analysis}*{Section 5.2})
    $$
    A_m:\mathscr{E}'(V,F)\to\mathscr{C}^\infty(U,E)~\text{is}~\text{continuous}
    $$
    or its Schwartz kernel ${A_m}(x,y)$ is smooth, i.e.~
    $$
    {A_m}(x,y)\in\mathscr{C}^\infty(U\times V, E\boxtimes F^*),
    $$
    where $E\boxtimes F^*$ is the vector bundle over $U\times V$ with fiber ${\rm Hom}(F_y,E_x)$ at each $(x,y)\in U\times V$.
    \item for any compact set $K$ in $U\times V$, multi-index
    $\alpha,\beta\in\mathbb{N}_0^N$ and $N\in\mathbb{N}_0$, where
    $\mathbb{N}_0:=\mathbb{N}\cup\{0\}$ and $\mathbb{N}:=\{1,2,3,\cdots\}$, there
    exists a constant $C_{K,\alpha,\beta,N}>0$ such that
    $$
    |\partial^\alpha_x\partial^\beta_y A_m(x,y)|\leq C_{K,\alpha,\beta,N} m^{-N}~\text{for all}~(x,y)\in K,~m~\text{large enough}.
    $$
\end{enumerate}
From now on, we also use the notation $A_m(x,y)=O(m^{-\infty})$ or $A_m=O(m^{-\infty})$ on $U\times V$ for an
$m$-negligible continuous linear operator $A_m$. Moreover, for two continuous linear
operators $A_m,B_m:\mathscr{C}^\infty_0(V,F)\to\mathscr{D}'(U,E)$, we write
$$
A_m(x,y)\equiv B_m(x,y)~{\rm mod}~O(m^{-\infty})~\text{on}~U\times V
$$
or
$$
A_m\equiv B_m~{\rm mod}~O(m^{-\infty})~\text{on}~U\times V
$$
if
$$
(A_m-B_m)(x,y)=O(m^{-\infty})~\text{on}~U\times V.
$$
Let $W$ be an open set in $\mathbb{R}^{N}$, we define the space
$$
S(1;W):=\{a\in\mathscr{C}^\infty(W):\sup_{x\in W}|\partial^\alpha_x a(x)|<\infty~\text{for all}~\alpha\in\mathbb{N}_0^N\}.
$$
Consider the space $S^0_{\rm loc}(1;W)$ containing all smooth functions $a(x,m)$ with
real parameter $m$ such that for all multi-index $\alpha\in\mathbb{N}_0^N$, any
cut-off function $\chi\in\mathscr{C}^\infty_0(W)$, we have
$$
\sup_{\substack{m\in\mathbb{R}\\m\geq 1}}\sup_{x\in W}|\partial^\alpha_x(\chi(x)a(x,m))|<\infty.
$$
For general $k\in\mathbb{R}$, we can also consider
$$
S^k_{\rm loc}(1;W):=\{a(x,m):m^{-k}a(x,m)\in\ S^0_{\rm loc}(1;W)\}.
$$
In other words, $S^k_{\rm loc}(1;W)$ takes all the smooth function $a(x,m)$ with parameter $m\in\mathbb{R}$ satisfying the estimate: for all $x\in W,~\text{multi-index}~\alpha\in\mathbb{N}_0^n~\text{and}~\chi\in\mathscr{C}^\infty_0(W)$ with ${\rm supp}(\chi)\subset K$, $K$ is a compact subset of $W$, then there is a constant $C_{K,\alpha}>0$ independent of $m$ such that
$$
|\partial^\alpha_x (\chi(x)a(x,m))|\leq C_{K,\alpha} m^k,~\text{for all}~m>0.
$$
For a sequence of $a_j\in S^{k_j}_{\rm loc}(1;W)$ with $k_j$ decreasing, $k_j\to-\infty$, and $a\in S^{k_0}_{\rm loc}(1;W)$. We say 
$$
a(x,m)\sim\sum_{j=0}^\infty a_j(x,m)~\text{in}~S^{k_0}_{\rm loc}(1;W)
$$
if for all $l\in\mathbb{N}_0$, we have
$$
a-\sum_{j=0}^l a_j\in S^{l}_{\rm loc}(1;W).
$$
In fact, for all sequence $a_j$ above, there always exists an element $a$ as the asymptotic sum, which is unique up to the elements in $S^{-\infty}_{\rm loc}(1;W):=\cap_{k} S^k_{\rm loc}(1;W)$. The above discussion can be found in \cite{DS1999}, and all the notations introduced above and can be generalized to the case on paracompact manifolds. 

Let $M$ be an open set in $\mathbb{R}^n$, $\dot{\mathbb{R}}^n:=\mathbb{R}^n\setminus\{0\}$, $\mathbb{N}:=\{1,2,3\cdots\}$ and $\mathbb{N}_0:=\mathbb{N}\cup\{0\}$. Let $\rho,\delta$ be real numbers such that $0\leq\delta<\rho\leq 1$.
\begin{definition}
\quad

\begin{enumerate}
    \item The symbol space with order $m$ in type $(\rho,\delta)$ on $M\times\mathbb{R}^N$ is denoted by $S^m_{\rho,\delta}(M\times\mathbb{R}^N)$, which is the space of all $a\in \mathscr{C}^{\infty}(M\times\mathbb{R}^N)$ satisfying: for all compact subset $K$ in $M$, and all multi-indices $\alpha\in \mathbb{N}_0^n,~\beta\in\mathbb{N}_0^N$, there exists a constant $C_{K,\alpha,\beta}>0$ such that
  $$
  \sup\limits_{(x,\theta)\in K\times\mathbb{R}^N}|\partial^{\alpha}_x\partial^{\beta}_{\theta}a(x,\theta)|\leq C_{K,\alpha,\beta}(1+|\theta|)^{m-\rho|\beta|+\delta|\alpha|}.
  $$
    \item A function $\phi(x,\theta)\in \mathscr{C}^{\infty}(M\times\mathbb{\dot R}^N)$ is called a phase function if it satisfies: $\text{Im}(\phi)\geq 0$, $\phi(x,\lambda\theta)=\lambda\phi(x,\theta)~\text{for all}~\lambda>0$ and every $(x,\theta)\in M\times\dot{\mathbb{R}}^{N}$, and
  $d\phi:=\sum\limits_{i=1}^n\frac{\partial\phi}{\partial x_i}dx_i+\sum\limits_{j=1}^N \frac{\partial\phi}{\partial{\theta_j}}d{\theta_j}\neq 0$ everywhere.
\end{enumerate}
\end{definition}
Let $a_j\in S^{m_j}_{\rho,\delta}(M\times\mathbb{R}^N),~j\in\mathbb{N}_0$ and $m_j\searrow{-\infty}$ as $j\to\infty$, then there exists $a\in S^{m_0}_{\rho,\delta}$ unique modulo $S^{-\infty}(M\times\mathbb{R}^N)$ such that
  $$
  a-\sum\limits_{0\leq j\leq k}a_j\in S^{m_k}_{\rho,\delta}~\text{for all}~k\in\mathbb{N}_0.
  $$
  We call such $a$ the  asymptotic sum of $a_j$, denoted by $a\sim\sum\limits_{j=0}^{\infty}a_j$. The space of classical symbols $S^m_{\rm cl}(M\times{\mathbb{R}^N})$   collects all $a(x,\theta)\in S^m_{1,0}(M\times{\mathbb{R}^N})$ such that
  $ a\sim\sum_{j=0}^\infty a_{m-j}(x,\theta)$, where the function $a_{m-j}\in\mathscr{C}^\infty(M\times{\mathbb{R}}^N)~\text{is positively homogeneous of degree}~m-j~\text{in the variable}~\theta\neq 0.$ Let the symbol $a\in S^m_{\rm cl}(M\times\mathbb{R}^N)$ and $\phi$ be a phase function on $X\times\dot{\mathbb{R}}^N$. If there is $k\in\mathbb{N}_0$ such that $m+k<-N$, then the oscillatory integral 
$$
I(a,\phi):=\int_{\mathbb{R}^N}e^{i\phi(x,\theta)}a(x,\theta)d\theta
$$
converges absolutely, and is a function in $\mathscr{C}^k(M)$ . Moreover, 
\begin{proposition}
\label{oscillatory integral}
    For any $m$ and $a\in S_{\rm cl}^{m}(M\times \mathbb{R}^N)$, there is a unique way of defining $I(a,\phi)\in\mathscr{D'}(M)$ such that
  \begin{enumerate}
    \item
    $I(a,\phi)=\int_{\mathbb{R}^N}e^{i\phi(x,\theta)}a(x,\theta)d\theta$ when $m<-N$.
    \item
    The map $a\mapsto {\rm}I(a,\phi)$ is continuous.
  \end{enumerate}
\end{proposition}
  For open sets $U\subset\mathbb{R}^{n_1},~V\subset\mathbb{R}^{n_2}$, by the Schwartz kernel theorem \cite{hormander2003analysis}*{Theorem 5.2.1}, there exists a bijection between $K_A\in\mathscr{D}'(U\times V)$ and a continuous linear map
$A:\mathscr{C}^{\infty}_0(V)\to\mathscr{D}'(U)$ by the correspondence
$$
\langle Au,v\rangle _{U}=\langle K_A,v\otimes u\rangle _{U\times V}
$$
for all $u\in \mathscr{C}^{\infty}_0(V),~v\in \mathscr{C}^{\infty}_0(U)$, where $\langle\cdot,\cdot\rangle$ means the pairing by duality and the tensor product $v\otimes u$ is defined $(v\otimes u)(x,y):=v(x)u(y)\in\mathscr{C}^\infty_0(U\times V)$. We call $K_A$ the distribution kernel of $A$. Let $\phi$ be a phase function on $(U\times V)\times\dot{\mathbb{R}^N}$, and
$a\in S^m_{\rm cl}\left((U\times V)\times\mathbb{R}^N\right)$. A continuous linear operator
$A:\mathscr{C}^{\infty}_0(V)\to\mathscr{D}'(U)$ is called a Fourier integral operator
if its distribution kernel is an oscillatory integral of the form
$$
K_A(x,y)=\int e^{i\phi(x,y,\theta)}a(x,y,\theta)d\theta.
$$
We formally write
$$
Au(x)=\iint e^{i\phi(x,y,\theta)}a(x,y,\theta)u(y)dyd\theta,~\text{for all}~u\in  \mathscr{C}^{\infty}_0(V).
$$
In particular, when $U=V$ and $\phi(x,y,\theta)=\langle x-y,\theta\rangle$, $A$ is said to be a pseudodifferential operator. We say a Fourier integral operator is properly supported if the projections $\pi_1:{\rm supp}(K_A)\to U$ and $\pi_2:{\rm supp}(K_B)\to V$ are proper maps. The discussion here can be found in \cite{GrigisSjostrand1994}, and all the notations introduced above can be generalized to the case on manifolds.

In the rest part of this section, we collect the essential tool in Melin--Sj\"ostrand theory on Fourier integral operators with complex phase \cite{MS1975}. For
$z=x+iy\in\mathbb{C}$, we write $\frac{\partial }{\partial z}:=\frac{1}{2}\left(\frac{\partial}{\partial x}+\frac{1}{i}\frac{\partial}{\partial y}\right)$ and $\frac{\partial}{\partial \overline{z}}=\frac{1}{2}\left(\frac{\partial}{\partial x}-\frac{1}{i}\frac{\partial}{\partial y}\right)$. Let $W\subset\mathbb{C}^n$ be an
open set, we say a $f\in\mathscr{C}^\infty (W)$ is almost analytic if for any compact
subset $K\subset W$ and any $N\in\mathbb{N_0}$, there is a constant $C_N>0$ such that
$$
\left|\frac{\partial f}{\partial\overline{z}}(z)\right|\leq C_N|{\rm Im}z|^N~\text{for all}~z\in K.
$$
We say two almost analytic functions $f_1$ and $f_2$ are equivalent, or $f_1\sim f_2$,
if for any compact subset $K\subset W$ and any $N\in\mathbb{N_0}$, there is a
constant $C_N>0$ such that
$$
|(f_1-f_2)(z)|\leq C_N|{\rm Im}z|^N~\text{for all}~z\in K.
$$
For $W_{\mathbb{R}}:=W\cap\mathbb{R}^n$, then for $f\in\mathscr{C}^\infty(W_{\mathbb{R}})$, $f$ always admits an almost analytic extension up to equivalence. One way is again via the Borel construction, for example, see \cite{Davies1995}*{Section 2.2}. The following proposition is about the critical point in the sense of Melin--Sj\"ostrand:
\begin{proposition}
\label{complex critical point}
Assume $f(x,w)$ is a smooth complex-valued function in a neighborhood of $(0,0)\in\mathbb{R}^{n+m}$ and that ${\rm Im}f\geq 0$, ${\rm Im}f(0,0)=0$, $f'_x(0,0)=0$, $\det f''_{xx}(0,0)\neq 0$. Let $\tilde{f}(z,w)$ be an almost analytic extension of f to a complex neighboehood of $(0,0)$, where $z=x+iy$ and $w\in\mathbb{C}^m$. By implicit function theorem, we denote $Z(w)$ to be the solution of
$$
\frac{\partial\tilde{f}}{\partial z}(Z(w),w)=0 
$$
in a neighborhood of of $0\in\mathbb{C}^m$. Then when $w$ is real, for every $N\in\mathbb{N}$, there is a constant $C_N>0$ such that for all $w\in\mathbb{R}^m$ near $0$,
$$
\left|\frac{\partial}{\partial w}\left(\tilde{f}(Z(w),w)\right)-\frac{\partial}{\partial w}\left(\tilde{f}(z,w)\middle|_{z=Z(w)}\right)\right|\leq C_N |{\rm Im}Z(w)|^N.
$$
Moreover, there are constants $C_1$, $C_2>0$ such that 
$$
{\rm Im}\tilde{f}(Z(w),w)\geq C_1|{\rm Im}Z(w)|^2,~w\in\mathbb{R}^m,~w~\text{near}~0
$$
and
$$
{\rm Im}\tilde{f}(Z(w),w)\geq C_2\inf_{x\in\Omega}\left({\rm Im}f(x,w)+|d_x f(x,w)|^2\right),~w\in\mathbb{R}^m,~w~\text{near}~0
$$
where $\Omega$ is an open set near the origin in $\mathbb{R}^n$. We call $\tilde{f}(Z(w),w)$ the corresponding critical value.
\end{proposition}

We end this part by the Melin--Sj\"ostrand complex stationary phase formula \cite{MS1975}*{Theorem 2.3}:
\begin{theorem}
\label{MS formula}
Let $f(x,w)$ be as in the Proposition \ref{complex critical point}. Then there are neighborhood $U$ and $V$ of the origin of $\mathbb{R}^n$ and $\mathbb{R}^m$, respectively, and differential operators $C_{f,j}$ in $x$ of order less equals to $2j$ with smooth coefficient of $w\in V$ such that
\begin{equation}
        \left|\int e^{it f(x,w)}u(x,w)dx-\left(\det\left(\frac{t \tilde{f}''_{zz}(Z(w),w)}{2\pi i}\right)\right)^{-\frac{1}{2}}e^{it\tilde{f}}(Z(w),w)\sum_{j=0}^{N-1}(C_{f,j}\tilde{u})(Z(w),w))t^{-j}\right|
\end{equation}
is bounded by $C_N t^{-N-\frac{n}{2}}$, where $C_N$ is a positive constant, $t\geq 1$ and $u\in\mathscr{C}^\infty_0(U\times V)$. Here, the function
$$
\left(\det\left(\frac{t \tilde{f}''_{zz}(Z(w),w)}{2\pi i}\right)\right)^{-\frac{1}{2}}
$$
is the branch of the square root of the
$$
\left(\det\left(\frac{t \tilde{f}''_{zz}(Z(w),w)}{2\pi i}\right)\right)^{-1}
$$
which is continuously deformed into $1$ under the homotopy
$$
s\in[0,1]\mapsto -i(1-s)\tilde{f}''_{zz}(Z(w),w)+sI\in{\rm GL}(n,\mathbb{C}).
$$
\end{theorem}
We note that all the discussion above can be generalize to the case on manifolds.
\section{Asymptotics for lower energy torus equivariant Szeg{\H{o}} kernel}
In this section we study the asymptotics for lower energy torus equivariant Szeg{\H{o}} kernel. First, we fix a number $\lambda>0$. Denote 
$$
\mathcal{H}^q_{b,\leq\lambda,m,mp_1,\cdots,mp_d}(X)=\bigoplus_{0\leq\mu\leq\lambda}\mathcal{H}^q_{b,\mu,m,mp_1,\cdots,mp_d}(X)
$$
where the space $\mathcal{H}^q_{b,\mu,m,mp_1,\cdots,mp_d}(X)$ is defined as in Theorem \ref{Hodge theorem for Kohn's Laplacian}. Apply Theorem \ref{Hodge theorem for Kohn's Laplacian}, it is clear that $\mathcal{H}^q_{b,\leq\lambda,m,mp_1,\cdots,mp_d}(X)$  is still a finite dimensional subspace of $\Omega^{(0,q)}_{m,mp_1,\cdots,mp_d}(X)$. Consider the spectral projector
$$
\Pi^{(q)}_{\leq\lambda,m,mp_1,\cdots,mp_d}:L^2_{(0,q)}(X)\to\mathcal{H}^q_{b,\leq\lambda,m,mp_1,\cdots,mp_d}(X).
$$
Denote $\Pi^{(q)}_{\leq\lambda,m,mp_1,\cdots,mp_d}(x,y)$ to be the distribution kernel of the spectral projector $\Pi^{(q)}_{\leq\lambda,m,mp_1,\cdots,mp_d}$. Let $N_q:=\dim_{\mathbb{C}}\mathcal{H}^q_{b,\leq\lambda,m,mp_1,\cdots,mp_d}(X)<\infty$ and $\{f_{\leq\lambda,j}^q\}_{j=1}^{N_q}$ be an orthonormal basis for the space $\mathcal{H}^q_{b,\leq\lambda,m,mp_1,\cdots ,mp_d}(X)$. Define the torus equivariant Szeg{\H{o}} kernel function on lower energy forms by
\begin{equation}
\label{low energy reproducing kernel at diagonal}
        {\rm Tr}\Pi^{(q)}_{\leq\lambda,m,mp_1,\cdots ,mp_d}(x)
        :=\sum_{j=1}^{N_q}|f_{\leq\lambda,j}^q(x)|_h^2:=\sum_{j=1}^{N_q}\left\langle f^q_{\leq\lambda,j}(x)\middle|f^q_{\leq\lambda,j}(x)\right\rangle
\end{equation}
where $\{f^q_{\leq\lambda,j}\}_{j=1}^{N_q}$ is an orthonormal basis for $\mathcal{H}^q_{b,\leq\lambda,m,mp_1,\cdots,mp_d}(X)$.
We note that the relation (\ref{reproducing kernel at diagonal}) and (\ref{diagonal kernel is kernel function}) in Section 2.1 also holds here: for an open set $D\subset X$, we take $\{e_j(x)\}_{j=1}^n$ varying smoothly in $x\in D$ such that $\{e_j(x)\}_{j=1}^n$ form an orthonormal basis of $T^{0,1}_x X$ for every $x\in D$. For a strictly increasing index set $J=\{j_1,\cdots,j_q\}$ with $|J|=q$, if we write $e^J:=e_{j_1}\wedge\cdots\wedge e_{j_q}$, then
\begin{equation}
                \Pi^{(q)}_{\leq\lambda,m,mp_1,\cdots ,mp_d}(x,y)
                ={\sum_{|I|=q}}'{\sum_{|J|=q}}'\Pi^{(q)}_{\leq\lambda,m,mp_1,\cdots ,mp_d,I,J}(x,y)e^I(x)\otimes\left(e^J(y)\right)^*
        \end{equation}
in the sense that
        \begin{equation}
                    \Pi^{(q)}_{\leq\lambda,m,mp_1,\cdots ,mp_d}u(x)={\sum_{|I|=q}}'{\sum_{|J|=q}}'\left(\int_D\Pi^{(q)}_{\leq\lambda,m,mp_1,\cdots ,mp_d,I,J}(x,y)u_J(y)dV_X(y)\right)e^I(x).
        \end{equation}
        for $u=\sum_{{|J|=q}}'u_J e^J\in\Omega^{(0,q)}(X)$. 
        We can check that
        \begin{equation}
        \label{low energy diagonal kernel is low energy kernel function}
             {\sum_{|J|=q}}'\left\langle\Pi^{(q)}_{\leq\lambda,m,mp_1,\cdots ,mp_d}(x,x) e^J(x)|e^J(x)\right\rangle={\rm Tr}\Pi^{(q)}_{\leq\lambda,m,mp_1,\cdots ,mp_d}(x).
        \end{equation}
        Here, $\Pi^{(q)}_{\leq\lambda,m,mp_1,\cdots ,mp_d,I,J}(x,y)\in\mathscr{C}^\infty(D\times D)$ for all strictly increasing index set $I,J$, $|I|=|J|=q$. Moreover, we can also check that for all strictly increasing index set $I,J$, $|I|=|J|=q$,
        \begin{equation}
        \label{low energy reproducing kernel in x and y}
                \Pi^{(q)}_{\leq\lambda,m,mp_1,\cdots ,mp_d,I,J}(x,y)
                =\sum_{j=1}^{N_q}f^q_{\leq\lambda,j,I}(x)\overline{f^q_{\leq\lambda,j,J}}(y),
        \end{equation}
where $f^q_{\leq\lambda,j}={\sum'_{|K|=q}}f^q_{\leq\lambda,j,K}e^K$, $j=1,\cdots,N_q$, is an orthonormal basis for $\mathcal{H}^q_{b,\leq\lambda,m,mp_1,\cdots,mp_d}(X)$. 
        
We divide our discussion of Szeg\H{o} kernel on space lower energy forms into the cases of the one away from $Y$ and the one near $Y$, where $Y:=\mu^{-1}(-p_1,\cdots,-p_d)$ define by the torus invariant CR moment map
$$
\mu:X\to\mathbb{R}^d,~\mu(x):=\bigg(\langle\omega_0(x),T_1(x)\rangle,\cdots,\langle\omega_0(x),T_d(x)\rangle\bigg),
$$
satisfying Assumption \ref{assumption 1}, \ref{assumption 2} and \ref{assumption 3}. On one hand, for the case away from $Y$, we can estimate the bound $\sup\limits_{\substack{x\in D\\D\cap Y=\emptyset}}|u(x)|$ for the functions $u\in\mathcal{H}^q_{b,\leq\lambda,m,mp_1,\cdots,mp_d}(X)$ with $\|u\|_X^2:=(u|u)=1$ by Bochner formula of $\Box^{(q)}_b$ and some standard PDEs argument. Combine with the relation (\ref{low energy reproducing kernel in x and y}), we can show that:
\begin{theorem}
\label{Main theorem 2}
For any open set $\Omega$ containing $Y$, 
$$
\Pi^{(q)}_{\leq\lambda,m,mp_1,\cdots,mp_d}(x,y)=O(m^{-\infty})
$$
on $(X\setminus\Omega)\times(X\setminus\Omega)$ if $q\in\{0,\cdots,n\}$.
\end{theorem}
On the other hand, for the case near $Y$, if we also assume the Levi form is positive near $Y$, then we can apply the Boutet--Sj\"ostrand type theorem \cite{HsiaoMarinsecu2017}*{Theorem 4.1}, i.e.~ in local picture $\Pi^{(q)}_{\leq\lambda}(x,y)$ is in the form of a complex phase Fourier integral operator, and we can study the asymptotic behavior of $\Pi^{(q)}_{\leq\lambda,m,mp_1,\cdots,mp_d}(x,y)$ using integration by parts and Melin--Sj\"ostrand complex stationary phase formula. Precisely, we have
\begin{theorem}
\label{Main theorem 3}
For each $p\in Y$, we can find a neighborhood denoted by $D_p$, such that 
$$
\Pi^{(q)}_{\leq\lambda,m,mp_1,\cdots,mp_d}(x,y)=O(m^{-\infty})
$$
on $D_p\times D_p$ if $q\in\{1,\cdots,n\}$. Finally, on $D_p\times D_p$,
$$
\Pi^{(0)}_{\leq\lambda,m,mp_1,\cdots,mp_d}(x,y)\equiv e^{im f(x,y)}b(x,y,m)~{\rm mod}~O(m^{-\infty}).
$$
Here, the phase function $f\in\mathscr{C}^\infty(D_p\times D_p)$ satisfies ${\rm Im}f\geq 0$, $f(x,x)=0$ for all $x\in Y\cap D_p$ and $d_x f(x,x)=-\omega_0(x)$, $d_y f(x,x)=\omega_0(x)$ for all $x\in Y\cap D_p$;
also, the symbol satisfies
\[
\begin{split}
\quad
&b(x,y,m)\in S^{n-\frac{d}{2}}_{{\rm loc}}(1;D_p\times D_p)\\
&b(x,y,m)\sim\sum_{j=0}^\infty m^{n-\frac{d}{2}-j}b_j(x,y)~\text{in}~S^{n-\frac{d}{2}}_{{\rm loc}}(1;D_p\times D_p)
\end{split}
\]
where $b_j(x,y)\in\mathscr{C}^\infty(D_p\times D_p)$, $j=0,1,2,\cdots$ and $b_0(x,x)>0$ for all $x\in Y\cap D_p$.
\end{theorem}
Finally, we show that the asymptotic behavior of kernel on lower energy forms Theorem \ref{Main theorem 2}, Theorem \ref{Main theorem 3} actually coincide with the one for genuine kernel Theorem \ref{Major theorem 1} as $m\to +\infty$. The precise statement is as followed:
\begin{theorem}
\label{Main theorem 4}
Under the same assumption in Theorem \ref{Major theorem 1}, then for any $\lambda>0$, as $m\to +\infty$
$$
\Pi^{(0)}_{m,mp_1,\cdots,mp_d}=\Pi^{(0)}_{\leq\lambda,m,mp_1,\cdots,mp_d}.
$$
\end{theorem}
\begin{proof}
Decompose the space
$$
\mathcal{H}^0_{b,\leq\lambda,m,mp_1,\cdots,mp_d}(X)=\mathcal{H}^0_{b,m,mp_1,\cdots,mp_d}(X)\oplus{\rm Span}_{0<\mu\leq\lambda} \mathcal{H}^0_{b,\mu,m,mp_1,\cdots,mp_d}(X).
$$
where
$$
\mathcal{H}^q_{b,\leq\lambda,m,mp_1,\cdots,mp_d}(X):=\bigoplus_{0\leq\mu\leq\lambda}\{u\in \Omega^{(0,q)}_{m,mp_1,\cdots,mp_d}(X):\Box_b^{(q)}u=\mu u\}.
$$ 
Apply Theorem \ref{Main theorem 3} and Theorem \ref{Main theorem 4} for the case $q=1$, then by the compactness of $X$ we know that 
$$
\dim_{\mathbb{C}}\mathcal{H}^1_{b,\leq\lambda,m,mp_1,\cdots,mp_d}(X)=\int_X {\rm Tr}\Pi^{(1)}_{\leq\lambda,m,mp_1,\cdots ,mp_d}(x)dV_X(x)\leq C_N m^{-N}.
$$
After fixing an $N\in\mathbb{N_0}$, we know that as $m\to +\infty$
$$
\mathcal{H}^1_{b,\leq\lambda,m,mp_1,\cdots,mp_d}(X)=\{0\}.
$$
 Since the group action is required to be CR, $T_j\overline{\partial}_b=\overline{\partial}_bT_j$ for all $j=0,1,\cdots,d$. Combine this with $\Box^{(q+1)}_b\overline{\partial}_b=\overline{\partial}_b \Box^{(q)}_b$, we can find that for any $u\in\mathcal{H}^0_{b,\mu,m,mp_1,\cdots,mp_d}(X)$, $0<\mu\leq\lambda$,
$$
\overline{\partial}_b u\in\mathcal{H}^1_{b,\leq\lambda,m,mp_1,\cdots,mp_d}(X).
$$
for $m$ large enough. However, this means that for some $0<\mu\leq\lambda$, for $m$ large enough
$$
\mu u=\Box^{(0)}_b u=\overline{\partial}_b^*\left(\overline{\partial}_b u\right)=0,
$$
i.e.~ ${\rm Span}_{0<\mu\leq\lambda} \mathcal{H}^0_{b,\mu,m,mp_1,\cdots,mp_d}(X)=\{0\}$ as $m\to +\infty$. Therefore, as $m\to +\infty$
$$
\mathcal{H}^0_{b,\leq\lambda,m,mp_1,\cdots,mp_d}(X)=\mathcal{H}^0_{b,m,mp_1,\cdots,mp_d}(X).
$$
Thus, for any $\lambda>0$, as $m\to +\infty$
$$
\Pi^{(0)}_{m,mp_1,\cdots,mp_d}=\Pi^{(0)}_{\leq\lambda,m,mp_1,\cdots,mp_d}.
$$
\end{proof}
\subsection{The asymptotic behavior away from $Y$}
In this section we prove Theorem \ref{Main theorem 2}. Fix a number $\lambda\geq 0$ and a point $p\notin Y$. Since $Y$ is closed, there always exist a neighborhood $D_p$ near $p$ and a $j\in\{1,\cdots,d\}$ such that
$$
\langle \omega_0(x), T_j(x)\rangle\neq -p_j~\text{for every}~x\in D_p.
$$
We may assume $j=1$, i.e.~$p_1+\langle\omega_0(x),T_1(x)\rangle\neq 0$ for all $x\in D_p$. Consider the vector field 
$$
F:=T_1+\langle\omega_0,T_1\rangle T_0,
$$
which is in $\mathscr{C}^{\infty}\left(X,T^{1,0}X\oplus T^{0,1}X \right)$ because $\langle\omega_0,F\rangle=0$. We hence decompose $F=L+\overline{L}$, where the vector field $L\in\mathscr{C}^{\infty}\left(X,T^{1,0}X\right)$. From now on, we assume $u\in\mathcal{H}^q_{b,\leq\lambda,m,mp_1,\cdots,mp_d}(X)$ with $\|u\|_X^2=1$. Take a cut-off function
$\chi\in\mathscr{C}^\infty_0(D_p)$ with $\chi=1$ near $p$. By Fourier inversion formula, we can see the multiplication operator $\chi\cdot$ as a properly supported zero order pseudodifferential operator. Precisely, for any $v\in\mathscr{C}^\infty_0(D_p)$
$$
\chi v(x)=\frac{1}{(2\pi)^{2n+1}}\iint e^{i\langle(x-y),\xi\rangle}\chi(x)v(y)dy d\xi.
$$
If we also regard $F$ as a first order differential operator, then
\begin{equation}
\label{kernel away from Y}
    \begin{split}
        F(\chi u)
        &=\chi(Fu)+[F,\chi]u\\
        &=\left(im(p_1+\langle\omega_0,T_1\rangle\right)\chi u+[F,\chi]u,
    \end{split}
\end{equation}
Cause we assume $p_1+\langle\omega_0,T_1\rangle\neq 0$ on $D_p$, from (\ref{kernel away from Y}) there are constants $C,C_0>0$ such that 
\begin{equation}
\label{a priori L2 estimate}
\begin{split}
\|\chi u\|_X^2
&\leq\frac{1}{C m^2}\left(\|F(\chi u)\|_X^2+\|[F,\chi]u\|_X^2\right)\\
&\leq\frac{1}{C m^2}\left(\|L(\chi u)\|_X^2+\|\overline{L}(\chi u)\|_X^2+C_0\|u\|_X^2\right),
\end{split}
\end{equation}
where $[\cdot,\cdot]$ denotes the commutator between differential operators and we see $[F,\chi]$ as an order $0$ properly supported pseudodifferential operator admitting $L^2$ continuity (See \cite{GrigisSjostrand1994}*{Theorem 3.6 and Theorem 4.5} for example). From now on, we use notations such as $C_0, C_1, C_2,\cdots$ or $C', C'',\cdots$ to denote positive constants independent of $m$. We need the following $L^2$-estimate to control (\ref{a priori L2 estimate}):
\begin{proposition}
\label{L^2 estimate of Kohn's laplacian}
For a given vector filed $Z\in\mathscr{C}^\infty(X,T^{1,0}X)$, there is a positive constant $C'=C'(Z)$ such that
$$
\|Z\phi\|_X^2+\|\overline{Z}\phi\|_X^2\leq C'\left(\left(\Box_b^{(q)}\phi\middle|\phi\right)+|(T_0\phi|\phi)|+\|\phi\|_X^2\right).
$$
for every $\phi\in\Omega^{(0,q)}_0(D_p)$.
\end{proposition}
\begin{proof}
First of all, for any $p\in X$, let $\{\overline{e}_j(x)\}_{j=1}^n$ varying smoothly
in $x$ near $p$ such that $\{\overline{e}_j(x)\}_{j=1}^n$ be an orthonomal frame of $T^{*0,1}_x X$ for all $x$ near $p$ and $\{\overline{L}_j\}_{j=1}^n$ be its dual frame on $T^{0,1}_xX$, $x$ is near $p$. Since $L_j$'s are first order differential operators, using integration by parts we can find the formal adjoint $\overline{L}_j^*$'s and $L_j^*$'s of $\overline{L}_j$'s and $L_j$ are
$$
\overline{L}_j^*=-L_j+E_j,~L_j^*=-\overline{L}_j+\overline{E}_j,
$$
respectively, where $E_j$ are some terms of zero order.
Now, we rewrite Theorem \ref{Bochner formula} into the form
$$
\Box^{(q)}_b=\sum_{j=1}^n \overline{L}_j^*\overline{L}_j+\sum_{j=1}^n a_j L_j+\sum_{j=1}^n b_j\overline{L}_j+c T_0+d,
$$
where $a_j$, $b_j$, $c_j$ , $c$ and $d$ are smooth coefficeints, $j=1,\cdots,n$. Then for $\phi\in\mathscr{C}^\infty_0(D_p)$, 
\begin{equation*}
    \left(\Box^{(q)}_b\phi\middle|\phi\right)
    \geq \sum_{j=1}^n\|\overline{L}_j\phi\|_X^2-\sum_{j=1}^n|(a_j L_j\phi|\phi)|-\sum_{j=1}^n|(b_j\overline{L}_j\phi|\phi)|-|(cT_0\phi|\phi)|-|(d\phi|\phi)|
\end{equation*}
Note that for any $\epsilon>0$, we have
\begin{equation*}
    \begin{split}
        |(a_j L_j\phi|\phi)|
        &\leq|(L_j a_j\phi|\phi)|+|([a_j,L_j]\phi|\phi)|\\
        &\leq |(a_j\phi|L_j^*\phi)|+C_1\|\phi\|_X^2\\
        &\leq |(a_j\phi|\overline{L}_j\phi)|+|(a_j\phi|\overline{E}_j\phi)|+C_1\|\phi\|_X^2\\
        &\leq \frac{C_2}{\epsilon}\|\phi\|_X^2+\epsilon\|\overline{L}_j\phi\|_X^2+C_3\|\phi\|_X^2+C_1\|\phi\|_X^2,
    \end{split}
\end{equation*}
and
$$
|(b_j\overline{L}_j\phi|\phi)|\leq C_4\epsilon\|\overline{L}_j\phi\|_X^2+\frac{1}{\epsilon}\|\phi\|_X^2 ,
$$
and
$$
|(cT_0\phi|\phi)|\leq C_5 \|T_0\phi\|\cdot\|\phi\|,~|(d\phi|\phi)|\leq C_6\|\phi\|_X^2.
$$
Take $\epsilon$ small enough, then we get
\begin{equation}
    \label{half Kohn's estimate 1}
\|\overline{L}_j\phi\|^2\leq C_7\left((\Box^{(q)}_b\phi|\phi)+|(T_0\phi|\phi)|+\|\phi\|_X^2\right),
\end{equation}
and sum over $j$ we have
\begin{equation}
    \label{half Kohn's estimate 1.1}
\|\overline{Z}\phi\|^2\leq C_7'\left((\Box^{(q)}_b\phi|\phi)+|(T_0\phi|\phi)|+\|\phi\|_X^2\right).
\end{equation}
So it remains to estimate $\|L_j\phi\|_X^2$ and $\|Z\phi\|_X^2$.
Observe that
\begin{equation*}
    \begin{split}
        \|L_j\phi\|_X^2
        &=|(L_j^*L_j\phi|\phi)|\\
        &\leq |(L_jL_j^*\phi|\phi)|+|([L_j^*,L_j]\phi|\phi)|\\
        &\leq \|L_j^*\phi\|_X^2+\left|\left(\left(\sum_{k=1}^n f_k L_k+\sum_{k=1}^n g_k \overline{L}_k)+ cT_0\right)\phi\middle|\phi\right)\right|\\
        &\leq \left(\|\overline{L}_j\phi\|_X^2+C_8\|\phi\|_X^2\right)+   C_9\left(\epsilon\|L\phi\|_X^2+\frac{1}{\epsilon}\|\phi\|_X^2+(\|\overline{L}\phi\|_X^2+\|\phi\|_X^2)+|(T_0\phi|\phi)|\right),
    \end{split}
\end{equation*}
where $f_k,~g_k,~c$ are some smooth coefficients. Choose $\epsilon$ small enough, sum over $j$, and apply (\ref{half Kohn's estimate 1}), then we can also get
\begin{equation}
\label{half Kohn's estimate 2}
    \|Z\phi\|_X^2\leq C_{10}\left((\Box^{(q)}_b\phi|\phi)+|(T_0\phi|\phi)|+\|\phi\|_X^2\right).
\end{equation}
Hence,
$$
\|Z\phi\|_X^2+\|\overline{Z}\phi\|_X^2\leq C'\left((\Box_b^{(q)}\phi|\phi)+|(T_0\phi|\phi)|+\|\phi\|_X^2\right).
$$
\end{proof}
Recall that from (\ref{a priori L2 estimate}) we have
\begin{equation}
\label{second a priori L2 estimate}
    \|\chi u\|_X^2
\leq\frac{1}{C m^2}\left(\|L(\chi u)\|_X^2+\|\overline{L}(\chi u)\|_X^2+C_0\right)
\end{equation}
and now Proposition \ref{L^2 estimate of Kohn's laplacian} implies that
\begin{equation}
\label{Kohn's type estimate}
\|L(\chi u)\|_X^2+\|\overline{L}(\chi u)\|_X^2+C_0\\
        \leq{C'}\left((\Box_b^{(q)}(\chi u)|\chi u)+(T_0(\chi u)|\chi u)+\|\chi u\|_X^2\right)+C_0.
\end{equation}
We now estimate all terms in (\ref{Kohn's type estimate}) one by one. First, we have
\begin{equation}
\label{The first term}
    \begin{split}
        (\Box_b^{(q)}(\chi u)|\chi u)
        &=(\chi\Box^{(q)}_b u|\chi u)+([\Box^{(q)}_b,\chi]u|\chi u)\\
        &\leq\lambda\|\chi u\|_X^2+\left((\sum_{j=1}^n c_j L_j+\sum_{j=1}^n d_j\overline{L}_j+eT_0)u|\chi u\right)\\
        &\leq\lambda\|\chi u\|_X^2+\left|\left((\sum_{j=1}^n c_j L_j+\sum_{j=1}^n d_j\overline{L}_j+eT_0)u|\chi u\right)\right|\\
        &\leq\lambda\|\chi u\|_X^2+\sum_{j=1}^n\left|\left(c_j L_j u|\chi u\right)\right|+\sum_{j=1}^n\left|\left(d_j \overline{L}_j u|\chi u\right)\right|+\left|\left(eT_0 u|\chi u\right)\right|
    \end{split}
\end{equation}
for some smooth coefficeints $c_j,~d_j,~e$, $j=1,\cdots,n$. Note that
\begin{equation}
    \label{T_0 term}
    \left|\left(eT_0 u|\chi u\right)\right|=|(eu|m\chi u)|\leq \frac{C_{11}}{\epsilon}\|u\|_X^2+\epsilon m^2\|\chi u\|_X^2;
\end{equation}
and for all $j=1,\cdots,n$,
\begin{equation}
\label{L_j term}
\begin{split}
            \left|\left(c_j L_j u|\chi u\right)\right|
            &=\left|(c_j u|L_j^*(\chi u))\right|+\left|([c_j,L_j]u|\chi u)\right|\\
            &\leq\left(\frac{C_{12}}{\epsilon}\|u\|_X^2+\epsilon\left(\|\overline{L}_j(\chi u)\|_X^2+C_{13}\|\chi u\|_X^2\right)\right)+\left(\frac{C_{14}}{\epsilon}\|u\|_X^2+\epsilon\|\chi u\|_X^2\right)\\
            &=\epsilon\|\overline{L}_j(\chi u)\|_X^2+(C_{13}\epsilon+\epsilon)\|\chi u\|_X^2+\frac{C_{12}+C_{14}}{\epsilon}\|u\|_X^2,
\end{split}
\end{equation}
and similarly
\begin{equation}
    \label{bar L_j term}
\left|\left(d_j \overline{L}_j u|\chi u\right)\right|\leq\epsilon\|L_j(\chi u)\|_X^2+(C_{16}\epsilon+\epsilon)\|\chi u\|_X^2+\frac{C_{15}+C_{17}}{\epsilon}\|u\|_X^2.
\end{equation}
Second, 
\begin{equation}
\label{The second term}
    \begin{split}
        |(T_0(\chi u)|\chi u)|
        &\leq|(\chi T_0 u|\chi u)|+|([T_0,\chi]u|\chi u)|\\
        &\leq m\|\chi u\|_X^2+\|[T_0,\chi]u\|_X\cdot\|\chi u\|_X\\
        &\leq m\|\chi u\|_X^2+\left(\frac{C_{18}}{\epsilon}\|u\|_X^2+\epsilon\|\chi u\|_X^2\right).
    \end{split}
\end{equation}
Therefore, $\|L(\chi u)\|_X^2+\|\overline{L}(\chi u)\|_X^2$ is bounded above by
\begin{equation*}
     C'\left(\left(m+(m^2+C_{19})\epsilon\right)\|\chi u\|_X^2
     +\epsilon\left(\|L_j (\chi u)\|_X^2+\|\overline{L}_j (\chi u)\|_X^2\right)
     +C_{20}\epsilon^{-1}\|u\|_X^2\right).
\end{equation*}
for some constant $C'>0$ independent of $m$ and $u$.
Take $\epsilon$ small enough and sum over $j$, then when $m$ large enough, there is also a constant $C''>0$ independent of $m$ and $u$ such that
$$
\|L(\chi u)\|_X^2+\|\overline{L}(\chi u)\|_X^2\leq C''\left((m+\epsilon m^2)\|\chi u\|_X^2+\epsilon^{-1}\|u\|_X^2\right)
$$
holds. Back to the estimate (\ref{second a priori L2 estimate}), i.e.~ 
$$
Cm^2\|\chi u\|_X^2\leq \|L(\chi u)\|_X^2+\|\overline{L}(\chi u)\|_X^2+C_0.
$$
Recall that we assume $\|u\|_X^2=1$ here, so if
we take a suitably small $\epsilon$ such that $\epsilon <\frac{C}{2C''}$, then when $m$ large enough we can find
\begin{proposition}
\label{kernel is small near Y}
For a point $p\in Y$ and $D_p$ a neighborhood of $p$ with $D_p\cap Y=\emptyset$, if we fix a number $\lambda\geq 0$, then for each function $\chi\in\mathscr{C}^\infty_0(D_p)$ and $q$-form $u\in\mathcal{H}^q_{b,\leq\lambda,m,mp_1,\cdots,mp_d}(X)$ with $\|u\|_X^2=1$, we can find a constant $C_{0,1}>0$ independent of $m$ and $u$ such that as $m\to +\infty$
\[
\|\chi u\|_X\leq\frac{1}{C_{0,1}m}.
\]
\end{proposition}
In fact, we can modify the above argument and improve the estimate Proposition \ref{kernel is small near Y}:
\begin{proposition}
    \label{kernel is small enough near Y}
Assume the same $p$, $D_p$, $u$ in Proposition \ref{kernel is small near Y}. For each $\chi\in\mathscr{C}^\infty_0(D_p)$ and any $N\in\mathbb{N}_0$, we can find a constant $C_{0,N}>0$ independent of $m$ and $u$ such that as $m\to +\infty$
  \[
\|\chi u\|_X\leq\frac{1}{C_{0,N} m^N}.
  \]  
\end{proposition}
To see this, we need to look back the estimate we did before, i.e. the one appeared in (\ref{a priori L2 estimate}), (\ref{Kohn's type estimate}), (\ref{The first term}) and (\ref{The second term}). First, take another cut-off function $\tau\in\mathscr{C}^\infty_0(D_p)$ with $\tau\equiv 1$ near ${\rm supp}(\chi)$,
observe that (\ref{a priori L2 estimate}) can be reformulated as
\begin{equation}
    \label{third a priori estimate}
    \begin{split}
       \|\chi u\|_X^2
    &\leq\frac{1}{C m^2}(\|F(\chi\tau u)\|_X^2+\|[F,\chi]\tau u\|)\\
    &\leq \frac{1}{C m^2}\left(\|L(\chi\tau u)\|_X^2+\|\overline{L}(\chi\tau u)\|_X^2+C_0\|\tau u\|_X^2\right).
    \end{split}
\end{equation}
Note that (\ref{T_0 term}) can be also viewed as
$$
\left|\left(e T_0 u|\chi u\right)\right|=\left|\left(e\tau u|m\chi u\right)\right|\leq\frac{C_0}{\epsilon}\|\tau u\|_X^2+\epsilon m^2\|\chi u\|_X^2.
$$
Similarly, we rewrite (\ref{L_j term}), (\ref{bar L_j term}) into
\begin{equation}
\begin{split}
            \left|\left(c_j L_j u|\chi u\right)\right|
            &=\left|\left(c_j\tau L_j u|\chi u\right)\right|\\
            &\leq\left|(c_j\tau u|L_j^*(\chi u))\right|+\left|([c_j\tau,L_j]u|\chi u)\right|\\
            &\leq\left(\frac{C_{12}}{\epsilon}\|\tau u\|_X^2+\epsilon\left(\|\overline{L}_j(\chi u)\|_X^2+C_{13}\|\chi u\|_X^2\right)\right)+{2C_{14}}\|u\|_X\cdot\|\chi u\|_X\\
            &=\epsilon\|\overline{L}_j(\chi u)\|_X^2+\frac{C_{12}}{\epsilon}\|\tau u\|_X^2+C_{13}\epsilon\|\chi u\|_X^2+2C_{14}\|u\|_X\cdot\|\chi u\|_X,
\end{split}
\end{equation}
and
\begin{equation}
  \left|\left(d_j \overline{L}_j u|\chi u\right)\right|\leq\epsilon\|{L}_j(\chi u)\|_X^2+\frac{C_{15}}{\epsilon}\|\tau u\|_X^2+C_{16}\epsilon\|\chi u\|_X^2+2C_{17}\|u\|_X\cdot\|\chi u\|_X,  
\end{equation}
respectively. Also, we take (\ref{The second term}) in the form of
\begin{equation}
    \begin{split}
        |(T_0(\chi u)|\chi u)|
        &\leq m\|\chi u\|_X^2+2C_{18}\|u\|_X\cdot\|\chi u\|_X.
    \end{split}
\end{equation}
Therefore, we have a slightly different upper bound 
\begin{equation*}
     C'\left(\left(m+(m^2+C_8)\epsilon\right)\|\chi u\|_X^2
     +\epsilon^{-1}\|\tau u\|_X^2
     +\epsilon\left(\|L_j (\chi u)\|_X^2+\|\overline{L}_j (\chi u)\|_X^2\right)
     +2C_{19}\|u\|_X\cdot\|\chi u\|_X\right).
\end{equation*}
for
$\|L(\chi u)\|_X^2+\|\overline{L}(\chi u)\|_X^2$. Take $\epsilon$ small enough and sum over $j$, then for all large enough $m$,
$$
\|L(\chi u)\|_X^2+\|\overline{L}(\chi u)\|_X^2\leq C''\left(\left(m+\epsilon m^2\right)\|\chi u\|_X^2
     +\epsilon^{-1}\|\tau u\|_X^2+2C_{19}\|u\|_X\cdot\|\chi u\|_X\right).
$$
From (\ref{third a priori estimate})
$$
Cm^2\|\chi u\|_X^2\leq\|L(\chi u)\|_X^2+\|\overline{L}(\chi u)\|_X^2+C_0\|\tau u\|_X^2,
$$
we can take sutibly small $\epsilon$ such that as $m$ large enough
\begin{equation}
\begin{split}
\label{cut off the result by tau}
        m^2\|\chi u\|_X^2
        &\leq C'''\left(\epsilon^{-1}\|\tau u\|_X^2+\|u\|_X\cdot\|\chi u\|_X+\|\tau u\|_X^2\right)\\
        &\leq C'''(\frac{1}{\epsilon C_{0,1}(\tau)m^2}+\frac{1}{C_{0,1}(\chi)m}+\frac{1}{C_{0,1}(\tau)m^2})\\
        &\leq \frac{1}{C_{0,\frac{3}{2}}m}.
\end{split}
\end{equation}
So $\|\chi u\|_X\leq\frac{1}{C_{0,\frac{3}{2}}m^{\frac{3}{2}}}$, and we can inductively apply (\ref{cut off the result by tau}) to get Proposition \ref{kernel is small enough near Y}.

Finally, for $p\in Y$, we take neighborhoods $O_p\Subset D_p$ of $p$ where $D_p\cap Y=\emptyset$, and pick a bump function $\chi\in\mathscr{C}^\infty_0(D_p)$ with $\chi\equiv 1$ on $O_p$. Denote $\|\cdot\|_k$ to be a torus invariant Sobolev $k$-norm induced by $(\cdot|\cdot)$. After applying the G\r{a}rding inequality to the $2k$-order strongly elliptic operator $(\Box^{(q)}_b-T_0^2)^k$, we have
$$
\|u\|_k^2\leq C_k'\left(\left((\Box^{(q)}_b-T_0^2)^k u\middle|u\right)+\|u\|_X^2\right)=O(m^{2k}).
$$
Similarly, with the help of elliptic estimate on $(\Box^{(q)}_b-T_0^2)^k$, for large enough $m$
\begin{equation}
    \begin{split}
\|\chi u\|_k^2
&\leq C_k'\left(\left((\Box^{(q)}_b-T_0^2)^k\chi u\middle|\chi u\right)+\|\chi u\|_X^2\right)\\
&\leq C_k'\left(\left(\chi(\Box_b^{(q)}-T_0^2)^k u\middle|\chi u\right)+\left(\left[(\Box_b^{(q)}-T_0^2)^k,\chi\right]u\middle|\chi u\right)+\|\chi u\|_X^2\right)\\
&\leq C_k''\left((\lambda+m^2)^k\|\chi u\|_X^2+\|u\|_{2k-1}\|\chi u\|_X+\|\chi u\|_X^2\right).
    \end{split}
\end{equation}
By Proposition \ref{kernel is small enough near Y}, for any $N\in\mathbb{N}_0$, there is a constant $C_{k,N}>0$ independent of $m$ and $u$ such that
\begin{equation}
    \label{sobolev k norm is very small}
    \|\chi u\|_k\leq C_{k,N}m^{-N},
\end{equation}
for all $m$ large enough. Combine (\ref{sobolev k norm is very small}) with Sobolev inequality, for all $x\in O_p$, $p\notin Y$, $k$ large enough, then for any $N\in\mathbb{N}$, there is a constant $C_N>0$ independent of $m$ and $u$ such that
\begin{equation}
\label{C0 estimate}
    |u(x)|_h^2=|\chi(x)u(x)|_h^2\leq C_k'''\|\chi u\|_k^2\leq C_N m^{-N},
\end{equation}
for all $m$ large enough. Now, consider a cut-off function $\tau\in\mathscr{C}^\infty_0(D_p)$, $\tau\equiv 1$ on ${\rm supp}(\chi)$. For any differential operator $P:\Omega^{(0,q)}_0(D_p)\to\Omega^{(0,q)}(D_p)$ of order $l$, where $\Omega^{(0,q)}_0(D_p):=\mathscr{C}^\infty_0(D_p,T^{*0,q}X)$. Note that
$$
Pu(x)=\chi(x)Pu(x)=P(\chi u)(x)+[\chi, P](\tau u)(x).
$$
Thus, similar in (\ref{C0 estimate}), for every $x\in O_p$, $p\notin Y$, $k$ large enough and any $N\in\mathbb{N}$, there is a constant $C_N>0$ independent of $m$ and $u$ such that
\begin{equation}
\label{Ck estimate}
    |Pu(x)|_h^2\leq C_k'''\left(\|P(\chi u)\|_k^2+\|[\chi,P](\tau u)\|_k^2\right)\leq C_k''''\left(\|\chi u\|_{k+l}^2+\|\tau u\|_{k+l-1}\right)\leq C_N m^{-N}
\end{equation}
for all $m$ large enough. From (\ref{C0 estimate}) and (\ref{Ck estimate}) and (\ref{low energy reproducing kernel in x and y}), Theorem \ref{Main theorem 2} holds. 
\subsection{The full asymptotic expansion near $Y$}
In this section, we prove Theorem \ref{Main theorem 3}. We first calculate the cirle equivariant Szeg\H{o} kernel. From now on, we fix a point $p\in Y$ and take a BRT patch $D$ near $p$ as in Proposition \ref{BRT patch} . Let 
$$
\Omega^{(0,q)}_m(X):=\left\{u\in\Omega^{(0,q)}(X):-iT_0 u:=mu\right\}
$$
and $L^2_{(0,q),m}(X)$ be the completion of $\Omega^{(0,q)}_m(X)$ with respect to $(\cdot|\cdot)$. With respect to $(\cdot|\cdot)$, denote $Q^{(q)}_m$ to be the orthogonal projection
$$
Q_m^{(q)}:L^2_{(0,q)}(X)\to L^2_{(0,q),m}(X).
$$
Extend $\Box_b^{(q)}$ by Gaffney extension (\ref{Gaffney extension}), then $\Box_b^{(q)}$ is a self-adjoint operator. We can hence apply generel theory for self-adjoint operator such as \cite{Davies1995} to  take the spectral projector
$$
\Pi_{\leq\lambda}^{(q)}:L^2_{(0,q)}(X)\to \mathcal{H}^{q}_{b,\leq\lambda}(X):=E((-\infty,\lambda]),
$$
for any $\lambda>0$, where $E((-\infty,\lambda])$ is the spectral projection and $E$ is the spectral measure for $\Box_b^{(q)}$, respectively. Denote the $m$-th Fourier component of the spectral projector by
$$
\Pi_{\leq\lambda,m}^{(q)}:L^2_{(0,q)}(X)\to \mathcal{H}^q_{b,m,\leq\lambda}(X):=\mathcal{H}^q_{b,\leq\lambda}(X)\cap L^2_{(0,q),m}(X).
$$
We can check that on $\Omega^{(0,q)}(X)$,
$$
\Pi^{(q)}_{\leq\lambda}=Q_m^{(q)}\Pi^{(q)}_{\leq\lambda}=\Pi^{(q)}_{\leq\lambda}Q_m^{(q)}.
$$
For a given point $p\in Y$, let $x=(x_1,\cdots,x_{2n},x_{2n+1})=(\mathring{x},x_{2n+1})$, $y=(y_1,\cdots,y_{2n},y_{2n+1})=(\mathring{y},y_{2n+1})$ be BRT trivialization as in Theorem \ref{BRT patch} defined on $D:=\tilde{D}\times(-\pi,\pi)\subset X$ near $p$, where $\tilde{D}$ is an open set of $\mathbb{C}^n$. Note that by theory of Fourier series on the circle, $Q_m^{(q)}u(x)=\frac{1}{2\pi}\int_{-\pi}^{\pi}e^{-im\theta}u(e^{i\theta}\circ x)d\theta$ for any $u\in\Omega^{(0,q)}(X)$. In particular, under BRT coordinates, for all $u\in\Omega^{(0,q)}_0(D):=\mathscr{C}^\infty_0(D,T^{*0,q}X)$,
\begin{equation*}
\begin{split}
        \Pi^{(q)}_{\leq\lambda,m}u(x)
        &=Q_m^{(q)}\Pi^{(q)}_{\leq\lambda}u(x)\\
    &=\frac{1}{2\pi}\int_{-\pi}^{\pi}e^{-im\theta}\left(\int_D \Pi^{(q)}_{\leq\lambda}(e^{i\theta}\circ x,y)u(y)dV_X(y)\right)d\theta\\
    &=\int_D\left(\frac{1}{2\pi}\int\Pi^{(q)}_{\leq\lambda}(e^{i\theta}\circ x,y)e^{-im\theta}d\theta\right)u(y)dV_X(y).
\end{split}
\end{equation*}
In other words,
\begin{equation}
    \label{original circle equivariant Szego kernel}
    \Pi^{(q)}_{\leq\lambda,m}(x,y)=\frac{1}{2\pi}\int_{-\pi}^{\pi}\Pi^{(q)}_{\leq\lambda}(e^{i\theta}\circ x,y)e^{-im\theta}d\theta.
\end{equation}
Similarly, for a fixed $(p_1,\cdots,p_d)\in\mathbb{Z}^d$ , we can find that  $\Pi^{(0)}_{m,mp_1\,\cdots,mp_d}u(x)$ is
\[
(2\pi)^{-d}\int_X\left(\int_{T^d}\Pi^{(0)}_{\leq\lambda,m}\left((e^{i\theta_1},\cdots,e^{i\theta_d})\circ x,y\right)e^{-im\sum_{j=1}^d p_j\theta_j}d\theta_1,\cdots d\theta_d \right)u(y)dV_X(y),
\]
for all $u\in\Omega^{(0,q)}(X)$.Therefore, 
\begin{equation}
\label{definition of torus equiv kernel}
\begin{split}
        \Pi_{\leq \lambda,m,mp_1,\cdots,mp_d}^{(0)}(x,y)
        &=\frac{1}{(2\pi)^d}\int_{T^d}\Pi^{(0)}_{\leq\lambda,m}\left((e^{i\theta_1},\cdots,e^{i\theta_d})\circ x,y\right)e^{-im\sum_{j=1}^d p_j\theta_j}d\theta_1,\cdots d\theta_d.
\end{split}
\end{equation}
Since we assume the Levi form is positive on $Y$, we can apply the result in \cite{HsiaoMarinsecu2017} for the case of constant signature $(n_-,n_+)=(0,n)$ near $Y$. On one hand, from \cite{HsiaoMarinsecu2017}*{Theorem 4.1}, we have:
\begin{proposition}
\label{HsiaoMawinescu for q not n_-}
For each $q=1,\cdots,n-1$, $\Pi^{(q)}_{\leq\lambda}$ is a smoothing operator near $Y$.     
\end{proposition}
Form Proposition \ref{HsiaoMawinescu for q not n_-}, using integration by parts with respect to $\theta$ in (\ref{original circle equivariant Szego kernel}), beacuse the boundary term vanishes for periodic reason, we can show that on $\Omega\times\Omega$,
$$
\Pi^{(q)}_{\leq\lambda,m}(x,y)=O(m^{-\infty})~\text{for all}~q=1,\cdots,n-1,
$$
where $\Omega$ is an open set containing $Y$. In particular, from (\ref{definition of torus equiv kernel}) and Theorem \ref{Main theorem 2}, we have:
\begin{proposition}
\label{near Y, q from 1 to n-1}
For each $q=1,\cdots,n-1$, $\Pi^{(q)}_{\leq\lambda,m,mp_1,\cdots,mp_d}(x,y)=O(m^{-\infty})$ on $X\times X$.
\end{proposition}
On the other hand, from the statement and the proof of \cite{HsiaoMarinsecu2017}*{Theorem 4.1}, we know:
\begin{proposition}
\label{HsiaoMarinescu for q is n_-}
For $q=\{0,n\}$, locally on a coordinates patch $D\subset X$, $\Pi^{(q)}_{\leq\lambda}$  is in the form of complex Fourier integral operator.  Precisely, in the sense of oscillatory integral
 $$
 \Pi_{\leq\lambda}^{(0)}(x,y)=\int_0^\infty e^{i\phi_-(x,y)t}a_-(x,y,t)dt.
 $$
 Moreover, for any small open neighborhood $\Omega$ containing $Y$ and all $\chi,\tau\in\mathscr{C}^\infty_0(\Omega)$ such that ${\rm supp}(\chi)\cap{\rm supp}(\tau)=\emptyset$, then $\chi\Pi^{(0)}_{\leq\lambda}\tau$ is a smoothing operator. Here the phase function locally on  $D$ is
 $$
 \phi_-(x,y)=x_{2n+1}-y_{2n+1}+\Phi(\mathring{x},\mathring{y}),
 $$
 where $\mathring{x}:=(x_1,\cdots,x_{2n})$, $\mathring{y}:=(y_1,\cdots,y_{2n})$ and $\Phi(\mathring{x},\mathring{y})$ is a complex-valued function satisfying
 for some constant $C>0$, ${\rm Im}\Phi(\mathring{x},\mathring{y})\geq C|\mathring{x}-\mathring{y}|^2$, $\Phi(\mathring{x},\mathring{y})=-\overline{\Phi(\mathring{y},\mathring{x})}$,    and $\Phi(\mathring{x},\mathring{y})=0~\text{if and only if}~\mathring{x}=\mathring{y}$, for all $(x,y)\in D\times D$.
And the symbol here satisfies 
 $$
 a_-(x,y,t)\in S^n_{\rm cl}(D\times D\times\mathbb{R}_+),~a_-(x,y,t)\sim\sum_{j=0}^\infty t^{n-j}(a_-)_j(x,y)~\text{in}~S^n_{\rm cl}(D\times D\times\mathbb{R}_+),
 $$
 where $(a_-)_j(x,y)\in\mathscr{C}^\infty(D\times D)$, $j=0,1,2,\cdots$, and $(a_-)_0(x,x)=\frac{1}{2\pi^{n+1}}|\det\mathcal{L}_x|$.
 
 Similarly,
  $$
 \Pi_{\leq\lambda}^{(n)}(x,y)=\int_0^\infty e^{i\phi_+(x,y)t}a_+(x,y,t)dt,
 $$
where the phase on $D$ is
 $$
 \phi_+(x,y)=-\overline{\phi}_-(x,y)=-x_{2n+1}+y_{2n+1}-\overline{\Phi}(\mathring{x},\mathring{y})
 $$
and the symbol 
 $$
 a_+(x,y,t)\in S^n_{\rm cl}(D\times D\times\mathbb{R}_+, T^{*0,n}X\boxtimes T^{*0,n}X).
 $$
 Also, for any small open neighborhood $\Omega$ containing $Y$ and all $\chi,\tau\in\mathscr{C}^\infty_0(\Omega)$ such that ${\rm supp}(\chi)\cap{\rm supp}(\tau)=\emptyset$, then $\chi\Pi^{(n)}_{\leq\lambda}\tau$ is a smoothing operator.
\end{proposition}
To calculate (\ref{original circle equivariant Szego kernel}) via Proposition \ref{HsiaoMarinescu for q is n_-}, we need to consider the integral under the BRT coordinates patch $D:=\tilde{D}\times (-\pi,\pi)$ in Theorem \ref{BRT patch}. First, note that under BRT coordinates, for $x\in D$ and $e^{i\theta}\circ x\in D$, we have $e^{i\theta}\circ x=(x_1,\cdots,x_{2n},x_{2n+1}+\theta)$. Second, to make sure all the calculation are under BRT patch, for $(x,y)\in D\times D$, we have to consider a smaller open set $D_p\subset D$ such that $\overline{D}_p\subset D$, and cut-off functions $\chi_0,\chi_1\in\mathscr{C}^\infty_0(D)$, where $\chi_0\equiv 1$ on $D_p$ and $\chi_1\equiv 1$ on ${\rm supp}(\chi)$. Notice that
$$
\Pi^{(0)}_{\leq\lambda,m}(x,y)=e^{imx_{2n+1}}\Pi^{(0)}_{\leq\lambda,m}(\hat{x},y),
$$
where $\hat{x}:=(x_1,\cdots,x_{2n},0)$. This holds because from (\ref{low energy reproducing kernel in x and y}),
$\Pi^{(0)}_{\leq\lambda,m}(x,y)=\sum_{j=1}^{N_0}f^0_{\leq\lambda,j}(x)\overline{f_{\leq\lambda,j}^0}(y)$ and $f^0_{\leq\lambda,j}(x)=e^{im x_{2n+1}}f^0_{\leq\lambda,j}(\hat{x})$
by $T_0=\frac{\partial}{\partial x_{2n+1}}$, where $\{f^0_{\leq\lambda,j}\}_{j=1}^{N_0}$ is an orthonormal basis for $\mathcal{H}^0_{b,\leq\lambda,m}(X)$. Now, for $(x,y)\in D_p\times D_p$, we write
\begin{equation}
\label{I_1+I_2}
    \begin{split}
         \Pi_{\leq\lambda,m}^{(0)}(x,y)
        &=e^{imx_{2n+1}}\Pi^{(0)}_{\leq\lambda,m}(\hat{x},y)\\
        &=\frac{e^{imx_{2n+1}}}{2\pi}\int_{-\pi}^{\pi}\Pi^{(0)}_{\leq\lambda}(e^{i\theta}\circ\hat{x},y)e^{-im\theta}d\theta\\
        &=\frac{e^{imx_{2n+1}}}{2\pi}\int_{-\pi}^{\pi}\Pi^{(0)}_{\leq\lambda}(e^{i\theta}\circ\hat{x},y)\chi_0(y)e^{-im\theta}d\theta\\
        &=I_1+I_2
    \end{split}
\end{equation}
where
\begin{equation}
    \label{I_1}
    I_1:=\frac{e^{imx_{2n+1}}}{2\pi}\int_{-\pi}^{\pi}\chi_1(e^{i\theta}\circ \hat{x})\Pi^{(0)}_{\leq\lambda}(e^{i\theta}\circ\hat{x},y)\chi_0(y)e^{-im\theta}d\theta
\end{equation}
and
\begin{equation}
\label{I_2}
    \begin{split}
    I_2
    &:=\frac{e^{imx_{2n+1}}}{2\pi}\int_{-\pi}^{\pi}(1-\chi_1)(e^{i\theta}\circ \hat{x},y)\Pi^{(0)}_{\leq\lambda}(e^{i\theta}\circ\hat{x},y)\chi_0(y)e^{-im\theta}d\theta\\
    &=\frac{e^{imx_{2n+1}}}{2\pi}\int_{-\pi}^{\pi}\left((1-\chi_1)\Pi^{(0)}_{\leq\lambda}\chi_0\right)(e^{i\theta}\circ \hat{x},y)e^{-im\theta}d\theta.
\end{split}
\end{equation}
In (\ref{I_2}), since $(1-\chi_1)\Pi^{(0)}_{\leq\lambda}\chi_0$ is a smoothing operator in view of Proposition \ref{HsiaoMarinescu for q is n_-}, we can apply integration by parts with respect to $\theta$. Because the boundary term vanishes for periodic reason, we can find that $I_2=O(m^{-\infty})$. As for $(\ref{I_1})$, we shall write
\begin{equation}
\label{circel equivariant Szego kernel on lower energy function}
    \begin{split}
        I_1
        &=\frac{e^{imx_{2n+1}}}{2\pi}\int_{-\pi}^{\pi}\chi_1(e^{i\theta}\circ \hat{x},y)\Pi^{(0)}_{\leq\lambda}(e^{i\theta}\circ\hat{x})\chi_0(y)e^{-im\theta}d\theta\\
        &=\frac{e^{imx_{2n+1}}}{2\pi}\int_{-\pi}^{\pi}\int_0^\infty e^{i\phi_-(e^{i\theta}\circ \hat{x},y)t}\chi_1(e^{i\theta}\circ \hat{x})a_-(e^{i\theta}\circ \hat{x},y,t)\chi_0(y)e^{-im\theta}dtd\theta\\
                &=\frac{e^{imx_{2n+1}}}{2\pi}\int_{-\pi}^{\pi}\int_0^\infty e^{i\phi_-((\mathring{x},\theta),y)t-im\theta}\chi_1(\mathring{x},\theta)a_-\big((\mathring{x},\theta),y,t\big)\chi_0(y)dt d\theta\\
  &=\frac{m e^{imx_{2n+1}}}{2\pi}\int_{-\pi}^{\pi}\int_0^\infty e^{im\psi_-(x,y,t,\theta)}\chi_1(\mathring{x},\theta)a_-\big((\mathring{x},\theta),y,mt\big)\chi_0(y)dtd\theta
    \end{split}
\end{equation}
where 
$$
\psi_-(x,y,t,\theta):=\bigg(\theta-y_{2n+1}+\Phi(\mathring{x},\mathring{y})\bigg)t-\theta.
$$
Similarly, on $D_P\times D_p$, we write $\Pi_{\leq\lambda,m}^{(n)}(x,y)=I_3+I_4$, where
\begin{equation}
    I_3:=\frac{m e^{imx_{2n+1}}}{2\pi}\int_{-\pi}^{\pi}\int_0^\infty e^{im\psi_+(x,y,t,\theta)}\chi_1(\mathring{x},\theta)a_+\big((\mathring{x},\theta),y,mt\big)\chi_0(y)dtd\theta,
\end{equation}
and
\begin{equation}
    I_4:=\frac{e^{imx_{2n+1}}}{2\pi}\int_{-\pi}^{\pi}\left((1-\chi_1)\Pi^{(n)}_{\leq\lambda}\chi_0\right)(e^{i\theta}\circ \hat{x},y)e^{-im\theta}d\theta=O(m^{-\infty}).
\end{equation}
Here,
$$
\psi_+(x,y,t,\theta):=-\bigg(\theta-y_{2n+1}+\overline{\Phi}(\mathring{x},\mathring{y})\bigg)t-\theta.
$$
We first handle the case for $q=n$:
\begin{proposition}
\label{q=n circle equiv kernel}
$\Pi^{(n)}_{\leq\lambda,m}(x,y)=O(m^{-\infty})~\text{on}~D_p\times D_p$.
\end{proposition}
\begin{proof}
Consider a cut-off function $\chi_2(\theta)\in\mathscr{C}^\infty_0(\mathbb{R})$, $\chi_2(\theta)\equiv 1$ when $|\theta|\leq\frac{\pi}{4}$ and $\chi_2(\theta)\equiv 0$ when $\frac{\pi}{2}\leq|\theta|<\pi$. Write
$$
I_3=I_5+I_6
$$
where $I_5$ has the integrand cut off by $\chi_2$ and $I_6$ is the one cut off by
$1-\chi_2$. Since $t\geq 0$, the term $\frac{\partial\psi_+}{\partial\theta}=-t-1\neq 0$ for all $\theta\in (-\pi,\pi)$, so we can write
$e^{im\psi_+}=\frac{\partial}{\partial\theta}\left(\frac{e^{im\psi_+}}{-im(t+1)}\right)$. Take $I_5=I_5'+I_5''$, where $I_5'$ is the integration taken over $0\leq t\leq 1$ and  $I_5''$ is the one taken over $t>1$. By using integration by parts with
respect to $\theta$, we can find both $I_5'=O(m^{-\infty})$ and $I_5''=O(m^{-\infty})$. Thus, $I_5=O(m^{-\infty})$. As for $I_6$, for the case
$\mathring{x}\neq\mathring{y}$, we have ${\rm Im}\psi_+={\rm Im}\Phi(\mathring{x},\mathring{y})>0$, so $\Pi^{(n)}_{\leq\lambda,m}(x,y)=O(m^{-\infty})$ by the elementary inequality that for any $m,N\in\mathbb{N_0}$, $m^N e^{-m}\leq C_N$ for some constant $C_N>0$. For the case
$\mathring{x}=\mathring{y}$, $\psi_+=-(\theta-y_{2n+1})t-\theta$. Notice that we may
assume $\theta-y_{2n+1}\neq 0$ on $I_6$ by taking the open set $D_p$ small enough.
Consider a cut-off function $\tau\in\mathscr{C}^\infty_0(\mathbb{R})$, $\tau(t)\equiv 1$ when $|t|\leq 1$ and $\tau(t)\equiv 0$ when $|t|\geq 2$. Set 
\begin{equation}
    \label{dydadic decomposition}
    \tau_j(t):=\tau(2^{-j}t)-\tau(2^{1-j}t),~j\in\mathbb{N},~\tau_0:=\tau.
\end{equation}
Note that $\sum_{j=0}^\infty \tau_j=1$ and 
\begin{equation}
    \label{dydadic decomposition property}
    2^{j-1}\leq|t|\leq 2^{j+1}~\text{for}~t\in{\rm supp}(\tau_j),~j>0.
\end{equation}
By the construction of oscillatory integral, for example see \cite{hormander2003analysis}*{Theorem 7.8.2}, in this case
\[
  I_6=\frac{e^{imx_{2n+1}}}{2\pi}\sum_{j=0}^\infty \int_{-\pi}^{\pi}\int_0^\infty e^{im\psi_+(x,y,t,\theta)}\tau_j(t)(1-\chi_2(\theta))\chi_1(\mathring{x},\theta)a_+\big((\mathring{x},\theta),y,mt\big)\chi_0(y)dtd\theta.  
\]
Decompose $I_6=I_6'+I_6''$, where
\begin{equation}
    \label{I_6'}
    I_6':=\frac{e^{imx_{2n+1}}}{2\pi}\int_{-\pi}^{\pi}\int_0^\infty e^{im\psi_+(x,y,t,\theta)}\tau_0(t)(1-\chi_2(\theta))\chi_1(\mathring{x},\theta)a_+\big((\mathring{x},\theta),y,mt\big)\chi_0(y)dtd\theta
\end{equation}
and
\begin{equation}
    \label{I_6''}
    I_6'':=\frac{e^{imx_{2n+1}}}{2\pi}\sum_{j=1}^\infty \int_{-\pi}^{\pi}\int_0^\infty e^{im\psi_+(x,y,t,\theta)}\tau_j(t)(1-\chi_2(\theta))\chi_1(\mathring{x},\theta)a_+\big((\mathring{x},\theta),y,mt\big)\chi_0(y)dtd\theta.  
\end{equation}
On one hand, in (\ref{I_6''}), because $e^{-im(-(\theta-y_{2n+1})t-\theta)}=\frac{\partial}{\partial t}\left(\frac{e^{-im(-(\theta-y_{2n+1})t-\theta)}}{-im(\theta-y_{2n+1})}\right)$, we can integration by parts with respect to $t$, and after combining (\ref{dydadic decomposition}), (\ref{dydadic decomposition property}) and $a_+(x,y,t)\in S^n_{\rm cl}(D\times D\times\mathbb{R}_+, T^{*0,n}X\boxtimes T^{*0,n}X)$, we can find that $I_6''=O(m^{-\infty})$. 
On the other hand, in (\ref{I_6'}), we can also integration by parts with respect to $t$; however, the boundary term appears at $t=0$. Fortunately, thanks to $\chi_1$ has compact support in $(-\pi,\pi)$ and $e^{-im(-(\theta-y_{2n+1})t-\theta)}=\frac{\partial}{\partial \theta}\left(\frac{e^{-im(-(\theta-y_{2n+1})t-\theta)}}{im(t+1)}\right)$, we can again apply integration with respect to $\theta$, and no boundary term will appear. In this way, we can also find $I_6'=O(m^{-\infty})$.
\end{proof}
In particular, from (\ref{definition of torus equiv kernel}) and Theorem \ref{Main theorem 2}, we find that:
\begin{proposition}
\label{near Y, q is n}
$\Pi^{(n)}_{\leq\lambda,m,mp_1,\cdots,mp_d}(x,y)=O(m^{-\infty})$ on $X\times X$.
\end{proposition}
Next, for the case $q=0$, take the point $p\in Y$ which we set in the beginning of this section, and we can find that at the point
$$
(x,y,t,\theta)=(p,p,1,0),
$$
there are
$$
{\rm Im}\psi_-=0,~\frac{\partial\psi_-}{\partial t}=0,~\frac{\partial\psi_-}{\partial\theta}=0,
$$
and
$$
\det \psi_-^{''}=\det
\begin{bmatrix}
    \frac{\partial^2\psi_-}{\partial t^2}                 &\frac{\partial^2\psi_-}{\partial\theta\partial t} \\
    \frac{\partial^2\psi_-}{\partial t\partial\theta}     & \frac{\partial^2\psi_-}{\partial\theta^2} \\
    \end{bmatrix}=\det
\begin{bmatrix}
    0       & -1 \\
    -1      & 0  \\
    \end{bmatrix}=-1\neq 0.
$$
Thus the point $(x,y,t,\theta)=(p,p,1,0)$ satisfies the assumption in Proposition \ref{complex critical point}. Moreover, when $(x,y)$ varies near $(p,p)$ we can have the Melin--Sj\"ostrand critical value in Proposition \ref{complex critical point}, because the system of equations
$$
\frac{\partial\tilde{\psi}_-}{\partial {\tilde{t}}}(x,y,\tilde{t},\tilde{\theta})=0
$$
and
$$
\frac{\partial\tilde{\psi}_-}{\partial {\tilde{\theta}}}(x,y,\tilde{t},\tilde{\theta})=0
$$
also has the solution
$$
(\tilde{t},\tilde{\theta})=\bigg(1,y_{2n+1}-\Phi(\mathring{x},\mathring{y})\bigg)\in\mathbb{C}^2
$$
where
$$
\tilde{\psi}_-(x,y,\tilde{t},\tilde{\theta})=\bigg(\tilde{\theta}-y_{2n+1}+\Phi(\mathring{x},\mathring{y})\bigg)\tilde{t}-\tilde{\theta}
$$
is an almost analytic extension of $\psi_-$ with respect to $(t,\theta)$. Therefore, we consider the decomposition also denoted by
$$
I_1=I_7+I_8,
$$
where $I_7$ is the one cut off by a bump function $\chi_3(t,\theta)\in\mathscr{C}^\infty_0(\mathbb{R}^2)$ and $I_8$ is the one cut off by
$1-\chi_3$. Here, $\chi_3$ satisfies $\chi_3\equiv 1$ near $(t,\theta)=(1,0)$, ${\rm supp}\left(\chi_3(t,\theta)\right)\subset [\frac{1}{2},\frac{3}{2}]\times[-\frac{\pi}{2},\frac{\pi}{2}]$. 

On one hand, similar to
the proof of Proposition \ref{q=n circle equiv kernel}, consider a cut-off function
$\tau\in\mathscr{C}^\infty_0(\mathbb{R})$, $\tau(t)\equiv 1$ when $|t|\leq 1$ and $\tau(t)\equiv 0$ when $|t|\geq 2$. We also set $\tau_j$ as in (\ref{dydadic decomposition}), $j\in\mathbb{N}$. Again, by taking $D_p$ small enough, we may assume that $\partial_t\psi_-:=\frac{\partial\psi_-}{\partial t}=\theta-y_{2n+1}+\Phi(\mathring{x},\mathring{y})\neq 0$ on $I_8$. Take the decomposition $I_8:=I_8'+I_8''$, where
\begin{equation}
    \label{I_8'}
I_8':=\frac{m e^{imx_{2n+1}}}{2\pi}\int_{-\pi}^{\pi}\int_0^\infty e^{im\psi_-(x,y,t,\theta)}\tau_0(t)(1-\chi_3)(t,\theta)\chi_1(\mathring{x},\theta)a_-\big((\mathring{x},\theta),y,mt\big)\chi_0(y)dtd\theta    
\end{equation}
and
\begin{equation}
    \label{I_8''}
    I_8'':=\frac{m e^{imx_{2n+1}}}{2\pi}\sum_{j=1}^\infty\int_{-\pi}^{\pi}\int_0^\infty e^{im\psi_-(x,y,t,\theta)}\tau_j(t)(1-\chi_3)(t,\theta)\chi_1(\mathring{x},\theta)a_-\big((\mathring{x},\theta),y,mt\big)\chi_0(y)dtd\theta,
\end{equation}
By $e^{im\psi_-}=\frac{\partial}{\partial t}\left(\frac{e^{im\psi_-}}{im\partial_t\psi_-}\right)$, for (\ref{I_8''}), we can take integration by parts with respect to $t$, and combine with (\ref{dydadic decomposition}), (\ref{dydadic decomposition property}) and $a_-(x,y,t)\in S^n_{\rm cl}(D\times D\times\mathbb{R}_+)$ to show that $I_8''=O(m^{-\infty})$. Also, since $\chi_1$ has support in $(-\pi,\pi)$, for (\ref{I_8'}), we can apply integration by parts with respect to $t$ and $\theta$ as in Proposition \ref{q=n circle equiv kernel} to show that $I_8'=O(m^{-\infty})$.

On the other hand, by the Melin--Sj{\"o}strand complex stationary phase formula, see Theorem \ref{MS formula}, up to an element in $O(m^{-\infty})$, the $I_7$ can be written into 
$$
\frac{m}{2\pi}e^{im(\tilde{\psi}_-(x,y,1,y_{2n+1}-\Phi(\mathring{x},\mathring{y}))+x_{2n+1})}A(x,y,m)=e^{im({x_{2n+1}-y_{2n+1}+\Phi(\mathring{x},\mathring{y})})}A(x,y,m),
$$
where 
\begin{equation}
\label{A(x,y,m)}
    A(x,y,m):=\frac{\tilde{\chi}_3(1,y_{2n+1}-\Phi(\mathring{x},\mathring{y}))\tilde{\chi}_1(x,y_{2n+1}-\Phi(\mathring{x},\mathring{y}))\tilde{a}_-((\mathring{x},y_{2n+1}-\Phi(\mathring{x},\mathring{y})),y,m)\chi_0(y)}{\det\left(\frac{m\tilde{\psi}''_{-}\left(x,y,1,y_{2n+1}-\Phi(\mathring{x},\mathring{y})\right)}{2\pi i}\right)^{\frac{1}{2}}},
\end{equation}
is in the symbol space $S^n_{\rm loc}(1;D_p\times D_p)$,
$\tilde{\chi}_3$, $\tilde{\chi}_1$ and $\tilde{a}_-$ is an almost analytic extensions of $\chi_3$, $\chi_1$ and $a_-$ in the varaible $(t,\theta)$, respectively, and $\tilde{\psi}''_{-}:=\begin{bmatrix}
    \frac{\partial^2\tilde{\psi}_-}{\partial \tilde{t}^2}                 &\frac{\partial^2\tilde{\psi}_-}{\partial\tilde{\theta}\partial \tilde{t}} \\
    \frac{\partial^2\tilde{\psi}_-}{\partial \tilde{t}\partial\tilde{\theta}}     & \frac{\partial^2\tilde{\psi}_-}{\partial\tilde{\theta}^2} \\
    \end{bmatrix}$. Note that by basic properties of symbol space, we have the expansion $A(x,y,m)\sim\sum_{j=0}^\infty m^{n-j}A_j(x,y)$ in $S^n_{\rm loc}(1;D_p\times D_p)$, where $A_j(x,y)\in\mathscr{C}^\infty(D_p\times D_p)$, and by the construction of almost analytic extension, we can find that
\begin{equation}
    \label{symbol A(x,x)}
    A_0(p,p)=(a_-)_0(p,p)=\frac{\det\mathcal{L}_p}{2\pi^{n+1}}.
\end{equation}
We sum up the discussion so far as the following local result:
\begin{proposition}
\label{asymptotic expansion of circle equiv Szego kernel}
Let $p\in Y$, $D$ be a BRT coordinates patch as in Theorem \ref{BRT patch} near $p$ and take an small enough $D_p\subset D$ such that $\overline{D}_p\subset D$. For any $(x,y)\in D_p\times D_p$, the $m$-th Fourier component of the Szeg\H{o} kernel on lower energy functions is
$$
\Pi_{\leq\lambda,m}^{(0)}(x,y)\equiv e^{im (x_{2n+1}-y_{2n+1}+\Phi(\mathring{x},\mathring{y}))}A(x,y,m)~{\rm mod}~O(m^{-\infty}).
$$
Here, $\mathring{x}:=(x_1,\cdots,x_{2n})$, $\mathring{y}:=(y_1,\cdots,y_{2n})$; the function $\Phi(\mathring{x},\mathring{y})$ is a complex-valued function satisfying:
 $$
 {\rm Im}\Phi(\mathring{x},\mathring{y})\geq C|\mathring{x}-\mathring{y}|^2,~\text{for some constant}~C>0,
 $$
 and
 $$
 \Phi(\mathring{x},\mathring{y})=-\overline{\Phi(\mathring{y},\mathring{x})},~\Phi(\mathring{x},\mathring{y})=0~\text{if and only if}~\mathring{x}=\mathring{y}
 $$
 for all $(x,y)\in D_p\times D_p$; and
$A(x,y,m)\in S^n_{\rm loc}(1;D_p\times D_p)$ with asymptotic expansion
$$
A(x,y,m)\sim\sum_{j=0}^\infty m^{n-j}A_j(x,y)~\text{in}~S^n_{\rm loc}(1;D_p\times D_p),~A_j(x,y)\in\mathscr{C}^\infty(D_p\times D_p),A_0(p,p)=\frac{\det\mathcal{L}_p}{2\pi^{n+1}}.
$$
\end{proposition}
Before preceding, note that in view of Proposition \ref{HsiaoMarinescu for q is n_-}, we see that $\Pi^{(0)}_{\leq\lambda}$ is smoothing away from the diagonal. From this observation, we can eother use integration by parts with respect to $\theta$ in (\ref{original circle equivariant Szego kernel}) for the case near $Y$ or apply Thoerem \ref{Main theorem 2} for the case away from $Y$ to show that :
\begin{proposition}
\label{circle equiv kernel is smoothing when x not y}
For $(q_1,q_2)\in X\times X$ such that $q_1$ and $q_2$ are not in the same $S^1$-orbit, then on $U\times V$, $\Pi_{\leq\lambda,m}^{(0)}(x,y)=O(m^{-\infty})$, where $U$ and $V$ are some open neighborhoods of $q_1$ and $q_2$, respectively.
\end{proposition}

Now, for $x\in D$, if $(e^{i\theta_1},\cdots,e^{i\theta_d})\circ x\in D$, let 
$$
x':=x'(\theta_1,\cdots,\theta_d):=(\mathring{x}',x_{2n+1}'):=(e^{i\theta_1},\cdots,e^{i\theta_d})\circ(\mathring{x},x_{2n+1}).
$$
Recall that
\begin{equation}
    \label{original torus equivariant kernel}
        \Pi_{\leq \lambda,m,mp_1,\cdots,mp_d}^{(0)}(x,y)
        =\frac{1}{(2\pi)^d}\int_{T^d}\Pi^{(0)}_{\leq\lambda,m}\left((e^{i\theta_1},\cdots,e^{i\theta_d})\circ x,y\right)e^{-im\sum_{j=1}^d p_j\theta_j}d\theta_1,\cdots d\theta_d.
\end{equation}
Let $p\in Y$ with BRT coordinates patch $D$ near $p$ as in Theorem \ref{BRT patch} and $(x,y)\in D\times D$. We first observe that $(e^{i\theta_1},\cdots,e^{i\theta_d})\circ x\notin D$, then 
$$
{\Pi^{(0)}_{\leq\lambda,m}\left((e^{i\theta_1},\cdots,e^{i\theta_d})\circ x,y\right)}=O(m^{-\infty}).
$$
One way to see this is by writing
$$
{\Pi^{(0)}_{\leq\lambda,m}\left((e^{i\theta_1},\cdots,e^{i\theta_d})\circ x,y\right)}=e^{imy_{2n+1}}{\Pi^{(0)}_{\leq\lambda,m}\left((e^{i\theta_1},\cdots,e^{i\theta_d})\circ x,\hat{y}\right)}.
$$
Clearly, $(e^{i\theta_1},\cdots,e^{i\theta_d})\circ x$ and $\hat{y}$ must not in the same $S^1$-orbit, otherwise there exists a $\mathring{\theta}\in(-\pi,\pi)$ such that $(e^{i\theta_1},\cdots,e^{i\theta_d})\circ x=e^{i\mathring{\theta}}\circ\hat{y}=(\mathring{y},\mathring{\theta})\in D$ leading to a contradiction. From Proposition \ref{circle equiv kernel is smoothing when x not y}, this implies that
$$
{\Pi^{(0)}_{\leq\lambda,m}\left((e^{i\theta_1},\cdots,e^{i\theta_d})\circ x,y\right)}=O(m^{-\infty}).
$$
In view of (\ref{original torus equivariant kernel}), for simplicity, we assume
\begin{equation}
    \label{assumption 1 on group action for BRT calculating torus equiv kernel}
   (e^{i\theta_1},\cdots,e^{i\theta_d})\circ x\in D~\text{for all}~
(\theta_1,\cdots,\theta_d)\in T^d~\text{and}~x\in D. 
\end{equation}
from now on. Moreover,
\begin{proposition}
Let $p\in Y$ and $D_p:=\tilde{D}_p\times(-\epsilon,\epsilon)$, where $\epsilon$ is a small number, as in Proposition \ref{asymptotic expansion of circle equiv Szego kernel}, for $(x,y)\in D_p\times D_p$, if $(e^{i\theta_1},\cdots,e^{i\theta_d})\circ x\notin D_p$, then $\Pi^{(0)}_{\leq\lambda,m}((e^{i\theta_1},\cdots,e^{i\theta_d})\circ x,y)=O(m^{-\infty})$.
\end{proposition}
\begin{proof}
If $(e^{i\theta_1},\cdots,e^{i\theta_d})\circ x$ and $y$ are in the same $S^1$-orbit, i.e. there is a $\mathring{\theta}\in(-\pi,\pi)$ such that
$(e^{i\theta_1},\cdots,e^{i\theta_d})\circ x=e^{i\mathring{\theta}}\circ\hat{y}=(\mathring{y},\mathring{\theta})\notin D_p$, then there must be $|\theta|>\epsilon>|y_{2n+1}|$. Take cut-off functions $\chi_0,\chi_1\in\mathscr{C}^\infty_0(D)$, where $\chi_0\equiv 1$ on $D_p$ and $\chi_1\equiv 1$ on ${\rm supp}(\chi)$. Pick $\tau\in\mathscr{C}^\infty_0(\mathbb{R})$, where $\tau(t)\equiv 1$ when $|t|\leq 1$ and $\tau(t)\equiv 0$ when $|t|\geq 2$, and set
$\tau_j(t):=\tau(2^{-j}t)-\tau(2^{1-j}t),~j\in\mathbb{N},~\tau_0:=\tau$. From (\ref{I_1+I_2}), (\ref{I_1}), (\ref{I_2}) and $(\ref{circel equivariant Szego kernel on lower energy function})$ we can find
\[
\begin{split}
    ~&
    \Pi^{(0)}_{\leq\lambda,m}((e^{i\theta_1},\cdots,e^{i\theta_d})\circ x,y)\\
    &=\Pi^{(0)}_{\leq\lambda,m}((\mathring{y},\theta),y)\\
    &=\frac{e^{im\mathring{\theta}}}{2\pi}\int_{-\pi}^{\pi}\int_0^\infty e^{im\big((\theta-y_{2n+1}+\Phi(\mathring{y},\mathring{y}))t-\theta\big)}\chi_1(\mathring{y},\theta)a_-((\mathring{y},\theta),y,t)\chi_0(y)dtd\theta\\
    &=\frac{e^{im\mathring{\theta}}}{2\pi}\int_{-\pi}^{\pi}\int_0^\infty e^{im\big((\theta-y_{2n+1})t-\theta\big)}\chi_1(\mathring{y},\theta)a_-((\mathring{y},\theta),y,t)\chi_0(y)dtd\theta\\
    &=\frac{e^{im\mathring{\theta}}}{2\pi}\int_{-\pi}^{\pi}\int_0^\infty e^{im\big((\theta-y_{2n+1})t-\theta\big)}\tau_0(t)\chi_1(\mathring{y},\theta)a_-((\mathring{y},\theta),y,t)\chi_0(y)dtd\theta\\
    &+\sum_{j=1}^\infty\frac{e^{im\mathring{\theta}}}{2\pi}\int_{-\pi}^{\pi}\int_0^\infty e^{im\big((\theta-y_{2n+1})t-\theta\big)}\tau_j(t)\chi_1(\mathring{y},\theta)a_-((\mathring{y},\theta),y,t)\chi_0(y)dtd\theta\\
    &=:I_9+I_{10}.
\end{split}
\]
Since $\theta-y_{2n+1}\neq 0$, we can apply integration by parts with respect to $t$ to show that $I_{10}=O(m^{-\infty})$; moreover, since $\chi_1$ has compact support in $\theta$, we can also apply integration by parts with respect to $\theta$ to show that the boundary term appeared in partial integration with respect to $t$ in $I_{9}$ is also $O(m^{-\infty})$.
So when $(e^{i\theta_1},\cdots,e^{i\theta_d})\circ x$ and $y$ are in the same $S^1$-orbit, $\Pi^{(0)}_{\leq\lambda,m}((e^{i\theta_1},\cdots,e^{i\theta_d})\circ x,y)=O(m^{-\infty})$. And if $(e^{i\theta_1},\cdots,e^{i\theta_d})\circ x$ and $y$ are not in the same $S^1$-orbit, this proposition follows from Proposition \ref{circle equiv kernel is smoothing when x not y}.
\end{proof}
Again, from  (\ref{original torus equivariant kernel}), for simplicity, we assume
\begin{equation}
    \label{assumption 2 on group action for BRT calculating torus equiv kernel}
   (e^{i\theta_1},\cdots,e^{i\theta_d})\circ x\in D_p~\text{for all}~
(\theta_1,\cdots,\theta_d)\in T^d~\text{and}~x\in D_p. 
\end{equation}
from now on. From (\ref{original torus equivariant kernel}), we can write
\begin{equation}
    \label{original torus equivariant kernel on D_p}
    \Pi^{(0)}_{\leq\lambda,m}((e^{i\theta_1},\cdots,e^{i\theta_d})\circ x,y)=\frac{1}{(2\pi)^d}\int_{T^d} e^{im\Psi(x,y,\theta_1,\cdots\theta_d)}{A}(x',y,m)d\theta_1\cdots d\theta_d
\end{equation}
where the phase function is
$$
\Psi(x,y,\theta_1,\cdots\theta_d):=x_{2n+1}'-y_{2n+1}+\Phi(\mathring{x}',w)-\sum_{j=1}^d p_j\theta_j,
$$
and the symbol
$$
B(x,y,\theta_1,\cdots,\theta_d,m):=A(x',y,m).
$$
We shall also notice that when $(\theta_1,\cdots,\theta_d)\neq 0$, there must be $\mathring{x}'\neq \mathring{x}$, otherwise we will have
$$
1\circ(e^{i\theta_1},\cdots,e^{i\theta_d})\circ x=e^{i\theta_0}\circ(1,\cdots,1)\circ x
$$
where $\theta_0:=x_{2n+1}'-x_{2n+1}~{\rm mod}~(-\pi,\pi)$, contradicting Assumption
\ref{assumption 2}. Thus,
$$
{\rm Im}\Psi(x,y,\theta_1,\cdots,\theta_d)={\rm Im}\Phi(\mathring{x}',y)>0~\text{for}~(\theta_1,\cdots,\theta_d)\neq 0.
$$
We hence consider a cut-off function $\chi(\theta_1,\cdots,\theta_d)\in\mathscr{C}^\infty_0(\mathbb{R}^d)$ such that $\chi=1$ near $(\theta_1,\cdots,\theta_d)=(0,\cdots,0)$. Write
$$
\frac{1}{(2\pi)^d}\int_{T^d} e^{im\Psi(x,y,\theta_1,\cdots\theta_d)}{A}(x',y,m)d\theta_1\cdots d\theta_d=I_9+I_{10}
$$
where we cut the integrand of $I_{11}$ by $\chi(\theta_1,\cdots,\theta_d)$ and the one for $I_{12}$ by $1-\chi(\theta_1,\cdots,\theta_d)$. Note that in the integrand of $I_{12}$ we have ${\rm Im}\Psi>0$, so $I_{12}=O(m^{-\infty})$ by the elementary inequality that for any $m,N\in\mathbb{N_0}$, $m^N e^{-m}\leq C_N$ for some constant $C_N>0$. As for the $I_{11}$, we shall apply the Melin--Sj\"ostrand stationary
phase formula Proposition \ref{complex critical point} to establish the asymptotic
expansion for torus equivariant Szeg\H{o} kernel. 

We now fix a point $p\in Y$ and a small enough BRT patch $D_p$ containing $p$. We claim
that the point $(x,y,\theta_1,\cdots,\theta_d)=(p,p,0,\cdots,0)$ satisfies the
requirement in the Proposition $\ref{complex critical point}$. To see this, first note
that in real coordinates $(x,y)=(\mathring{x},x_{2n+1},\mathring{y},y_{2n+1})$ we have
$$
\frac{\partial\Psi}{\partial\theta_j}=-\frac{\partial x_{2n+1}'}{\partial\theta_j}+\frac{\partial\Phi(\mathring{x}',\mathring{y})}{\partial \mathring{x}'}\frac{\partial \mathring{x}'}{\partial\theta_j}-p_j.
$$
Under the local expression of the phase function \cite{HsiaoMarinsecu2017}*{Theorem 3.4} in terms of canonical coordinates, we have relations
\begin{equation}
    \label{phase funciton and contact form}
    \phi_-(x,y)=x_{2n+1}-y_{2n+1}+\Phi(\mathring{x},\mathring{y})~\text{and}~d_x{\phi_-(x,y)}|_{x=y}=-\omega_0(x).
\end{equation}
We can hence get
\begin{equation}
    \label{omega_0 at p}
    -\omega_0(p)=\left.\left(\frac{\partial\phi_{-}(x',y)}{\partial x'}dx'\right)\middle|\right._{\substack{x=y=p\\\theta=0}}=\left(-dx_{2n+1}+\frac{\partial\Phi(\mathring{x}',\mathring{y})}{\partial \mathring{x}'}d\mathring{x}\right)(p,p,0)
\end{equation}
and for $j=1,\cdots,d$,
\begin{equation}
\label{T_0 at p}
    \begin{split}
        T_j(p)
        &:=\left.\frac{\partial}{\partial\theta_j}\left((1,\cdots,e^{i\theta_j},\cdots,1)\circ x\right)\middle|\right._{\substack{x=p\\ \theta_j=0}}\\
        &=\left.\frac{\partial}{\partial\theta_j}\left((e^{i\theta_1},\cdots,e^{i\theta_j},\cdots,e^{i\theta_d})\circ x\right)\middle|\right._{\substack{x=p\\\theta_1=\cdots=\theta_d=0}}\\
        &=\left.\frac{\partial x'}{\partial\theta_j}\frac{\partial}{\partial x}\middle|\right._{\substack{x=p\\ \theta_1=\cdots=\theta_d=0}}\\
        &=\left(\frac{\partial x_{2n+1}'}{\partial\theta_j}\frac{\partial}{\partial x_{2n+1}}+\frac{\partial \mathring{x}'}{\partial\theta_j}\frac{\partial}{\partial \mathring{x}}\right)(p,0).
    \end{split}
\end{equation}
By the definition of  $Y$ and (\ref{omega_0 at p}) and (\ref{T_0 at p}), we know that for each $j=1,\cdots,d$,
\[
p_j=\langle-\omega_0(p),T_j(p)\rangle=\left(-\frac{\partial x_{2n+1}'}{\partial\theta_j}+\frac{\partial\Phi(\mathring{x}',\mathring{y})}{\partial \mathring{x}'}\frac{\partial \mathring{x}'}{\partial\theta_j}\right)(p,p,0).
\]
Thus
\[
\frac{\partial\Psi}{\partial\theta_j}(p,p,0)=0~\text{for all}~j=1,\cdots,d.
\]
Now, for $j,k=1,\cdots,d$, we need to show that
$$
\left(\frac{\partial^2\Psi}{\partial\theta_j\partial\theta_k}\right)_{j,k=1}^d=\left(-\frac{\partial^2 x_{2n+1}'}{\partial\theta_j\partial\theta_k}+\frac{\partial^2\Phi}{\partial \mathring{x}'^2}\frac{\partial \mathring{x}'}{\partial\theta_j}\frac{\partial \mathring{x}'}{\partial\theta_k}+\frac{\partial\Phi}{\partial \mathring{x}'}\frac{\partial^2 \mathring{x}'}{\partial\theta_j\partial\theta_k}\right)_{j,k=1}^d,
$$
where $\frac{\partial^2\Phi}{\partial \mathring{x}'^2}\frac{\partial \mathring{x}'}{\partial\theta_j}\frac{\partial \mathring{x}'}{\partial\theta_k}:=\sum_{a,b=1}^{2n}\frac{\partial^2\Phi}{\partial x'_a\partial x'_b}\frac{\partial x'_a}{\partial\theta_j}\frac{\partial x'_b}{\partial\theta_k}$ and $\frac{\partial\Phi}{\partial \mathring{x}'}\frac{\partial^2 \mathring{x}'}{\partial\theta_j\partial\theta_k}:=\sum_{c=1}^{2n}\frac{\partial\Phi}{\partial x'_c}\frac{\partial^2 x'_c}{\partial\theta_j\partial\theta_k}$, $j,k=1,\cdots,d$, is a non-singular matrix at $(x,y,\theta_1,\cdots,\theta_d)=(p,p,0,\cdots,0)$. Under BRT
coordinates, see Theorem \ref{BRT patch}, write $(z,w)=(z_1,\cdots,z_n,w_1,\cdots,w_n)$,
where $z_j:=x_{2j_1}+i x_{2j}$ and $w_j:=y_{2j-1}+y_{2j}$. In
\cite{HsiaoMarinsecu2017}*{Theorem 3.6}, it suggests that the term
$\Phi(\mathring{x},\mathring{y})$ in the phase function is in the form of
\begin{equation}
    \label{function Phi under BRT}
    \Phi(\mathring{x},\mathring{y})=i(\phi(z)+\phi(w))-2i\sum_{|\alpha|+|\beta|\leq N} \frac{\partial^{\alpha+\beta}\phi}{\partial z^\alpha\partial\overline{z}^\beta}(0)\frac{z^\alpha}{\alpha !}\frac{\overline{w}^\beta}{\beta !}+O(|(z,w)|^{N+1}),
\end{equation}
for every $N\in\mathbb{N}$, where the function $\phi$ is given by $\phi(z)=\sum_{j=1}^n \lambda_j|z_j|^2+O(|z|^3)$ and $\lambda_1,\cdots,\lambda_n>0$ are the eigenvalues of the Levi form $\mathcal{L}_p$ at the point $p$. Hence, at $(x,y,\theta_1,\cdots,\theta_d)=(p,p,0,\cdots,0)$, we have
$$
\frac{\partial\Phi}{\partial \mathring{x}'}(p,p,0)=0.
$$
Observe an easy fact from linear algebra: if $A$ and $B$ are real symmetric matrix, and $B$ is positive definite, then $C:=A+iB$ is non singular. (Consider the orthogonal decomposition $B=P^tP$, and $Q:=P^{-1}$, then $Q^t C Q= Q^t A Q+i$Id. Suppose $\det C=0$, then $-i$ is an eigenvalues of $A$, contradicting the fact that all the eigenvalues of $A$ are real). So it remains to show that
$$
{\rm Im}\left(\frac{\partial^2\Phi}{\partial \mathring{x}'^2}\frac{\partial \mathring{x}'}{\partial\theta_j}\frac{\partial \mathring{x}'}{\partial\theta_k}\right)(p,p,0)~\text{is positive definite}.
$$
Since $\Phi(\mathring{x},\mathring{y})$ has leading term $\sum_{j=1}^n\lambda_j|z_j-w_j|^2$ where $\lambda_j>0$, $j=1,\cdots,d$, we know that 
$$
{\rm Im}\frac{\partial^2\Phi}{\partial \mathring{x}'^2}(p,p,0)~\text{is positive definite}.
$$
To examine whether the submatrix
$$
{\rm Im}\left(\frac{\partial^2\Phi}{\partial \mathring{x}'^2}\frac{\partial \mathring{x}'}{\partial\theta_j}\frac{\partial \mathring{x}'}{\partial\theta_k}\right)(p,p,0)
$$
is positive definite, it sufficies to take any $0\neq u\in M_{d\times 1}(\mathbb{R})$ and $D:=\begin{bmatrix}
\frac{\partial {x}'_1}{\partial\theta_1} &\cdots & \frac{\partial {x}'_1}{\partial\theta_d}\\
\vdots &~ &\vdots\\
\frac{\partial {x}'_{2n}}{\partial\theta_1} &\cdots &\frac{\partial {x}'_{2n}}{\partial\theta_d}
\end{bmatrix}(p,0)$ and check that
$$
u^t{\rm Im}\left(\frac{\partial^2\Phi}{\partial \mathring{x}'^2}\frac{\partial \mathring{x}'}{\partial\theta_j}\frac{\partial \mathring{x}'}{\partial\theta_k}\right)(p,p,0)u=\left((Du)^t{\rm Im}\frac{\partial^2\Phi}{\partial \mathring{x}'^2}(Du)\right)(p,p,0)>0
$$
which is equivalent to examine whether $D$ is of full rank. And this is in fact guaranteed by assumption that the torus action is free near $Y$ and the assumption that $p\in Y$ is a regular level set. One way to see this is to take the map $\sigma_x: T_e T^d\to T_x X$ by $\frac{\partial}{\partial\theta_j}\mapsto T_j=\left.\frac{\partial}{\partial\theta_j}\middle|\right._{\theta_j=0}(1,\cdots,e^{i\theta_j},\cdots,1)\circ x$. For the torus action is free, i.e.~$\{g\in T^d: g\circ x=x, x~\text{near}~Y\}=\{e\}$, the map $\sigma_x$ is injective. So the column vectors 
\[
[T_1,\cdots,T_j](p)=
\begin{bmatrix}
\frac{\partial x_1'}{\partial \theta_1} & \cdots & \frac{\partial x_1'}{\partial \theta_d}\\
\vdots & &\vdots\\
\frac{\partial x_{2n}'}{\partial \theta_1} & \cdots & \frac{\partial x_{2n}'}{\partial \theta_d}\\
\frac{\partial x_{2n+1}'}{\partial \theta_1} & \cdots & \frac{\partial x_{2n+1}'}{\partial \theta_d}
\end{bmatrix}(p,0)
\]
has rank $d$. Also, since $p$ is in a regular level set, we know that 
\[
[d(T_1\lrcorner\omega_0),\cdots,d(T_d\lrcorner\omega_0)](p)=
\begin{bmatrix}
\frac{\partial\langle \omega_0,T_1\rangle_x}{\partial x_1} & \cdots & \frac{\partial\langle \omega_0,T_d\rangle_x}{\partial x_1}\\
\vdots & &\vdots\\
\frac{\partial\langle \omega_0,T_1\rangle_x}{\partial x_{2n}} & \cdots & \frac{\partial\langle \omega_0,T_d\rangle_x}{\partial x_{2n}}\\
\frac{\partial\langle \omega_0,T_1\rangle_x}{\partial x_{2n+1}} & \cdots & \frac{\partial\langle \omega_0,T_d\rangle_x}{\partial x_{2n+1}}
\end{bmatrix}(p)
\]
is of rank $d$. Since the one form $\omega_0$ is torus invariant, we know that $\mathcal{L}_{T_j}\omega_0=0$ for all $j=1,\cdots,d$. So the Cartan formula $\mathcal{L}_{T_j}\omega_0=T_j\lrcorner d\omega_0+d(T_j\lrcorner\omega_0)$
suggests that the one forms $T_j\lrcorner d\omega_0(p)\neq 0$ for all $j=1,\cdots,d$, otherwise for all $j=1,\cdots,d$, $d(T_j\lrcorner\omega_0)(p)=0$, leading to a contradiction. We also note that $T_0\lrcorner d\omega_0\equiv 0$. One way to see this is that under BRT coordinates Theorem \ref{BRT patch} $T_0=\frac{\partial}{\partial x_{2n+1}}$ and $d\omega_0$ is independent of $x_{2n+1
}$. Write
$$
T_j(p)=\sum_{k=1}^{2n+1}\frac{\partial x_k'}{\partial\theta_j}(p,0)\frac{\partial}{\partial x_k},
$$
and if we suppose that
$$
B:=
\begin{bmatrix}
\frac{\partial x_1'}{\partial \theta_1} & \cdots & \frac{\partial x_1'}{\partial \theta_d}\\
\vdots & &\vdots\\
\frac{\partial x_{2n}'}{\partial \theta_1} & \cdots & \frac{\partial x_{2n}'}{\partial \theta_d}
\end{bmatrix}(p,0)
$$
has rank less than $d$, we can find numbers $(\alpha_1,\cdots,\alpha_d)\in\dot{\mathbb{R}}^d$ such that $\sum_{j=1}^d\alpha_k\frac{\partial x_k'}{\partial\theta_j}(p,0)=0$
for all $k=1\cdots,2n$ and $\sum_{j=1}^d \alpha_j\frac{\partial x_{2n+1}'}{\partial\theta_j}(p,0)\neq 0$ (recall $[T_1,\cdots,T_d](p)$ has rank $d$). However, this means that $T_0(p)=\frac{\sum_{j=1}^d\alpha_j T_j(p)}{\sum_{j=1}^d \alpha_j\frac{\partial x_{2n+1}'}{\partial\theta_j}(p,0)}$, which leads to a
contradiction $\left(T_0\lrcorner d\omega_0\right)(p)=0$. So the matrix
$D:=\begin{bmatrix}
\frac{\partial {x}'_1}{\partial\theta_1} &\cdots & \frac{\partial {x}'_1}{\partial\theta_d}\\
\vdots &~ &\vdots\\
\frac{\partial {x}'_{2n}}{\partial\theta_1} &\cdots &\frac{\partial {x}'_{2n}}{\partial\theta_d}
\end{bmatrix}(p,0)$ has full rank $d$, and we conclude that
$$
\det\left(\frac{\partial^2\Psi}{\partial\theta_j\partial\theta_k}\right)_{j,k=1}^d(p,p,0)\neq 0.
$$
The above discussion implies that the point $(x,y,\theta_1,\cdots,\theta_d)=(p,p,0,\cdots,0)$ satisfies the assumption in Proposition \ref{complex critical point}, where $p\in Y$ is fixed. 

Let $\tilde{\theta}_1(x,y),\cdots,\tilde{\theta}_d(x,y)$ be the solution of the equations 
$$
\left.\frac{\partial\tilde{\Psi}}{\partial\tilde{\theta}_j}\middle(x,y,\tilde{\theta}_1(x,y),\cdots,\tilde{\theta}_d(x,y)\right)=0,~j=1,\cdots,d,
$$
in a neighborhood of $(x,y)=(p,p)\in\mathbb{C}^{2n}$ with $\tilde{\theta}_j(x,y)=0$ at $(x,y)=(p,p)$ for all $j=1,\cdots,d$, where $\tilde{\Psi}$ ia an almost analytic extension of $\Psi$ in the variable $\theta_j$'s and $\tilde{\theta}_j$'s are allowed to be complex near $(0,\cdots,0)\in\mathbb{C}^d$. Accordingly, by complex stationary phase formula Proposition \ref{MS formula} and the basic properties of almost analytic extension, up to an element in $O(m^{-\infty})$, the torus equivariant Szeg\H{o} kernel $\Pi^{(0)}_{\leq\lambda,m,mp_1,\cdots,mp_d}(x,y)$ is
\begin{equation}
\label{coarse torus equivariant kernel}
    \begin{split}
    ~& (2\pi)^{-d}e^{im\tilde{\Psi}(x,y,\tilde{\theta}_1(x,y),\cdots,\tilde{\theta}_d(x,y))}\left(\det\frac{m\tilde{\Psi}''_{\tilde{\theta}\tilde{\theta}}(x,y,\tilde{\theta}_1(x,y),\cdots,\tilde{\theta}_d(x,y))}{2\pi i}\right)^{-\frac{1}{2}}\\
        &\tilde{\chi}\left(\tilde{\theta}_1(x,y),\cdots,\tilde{\theta}_d(x,y)\right) \tilde{A}(\widetilde{x'}(\tilde{\theta}_1(x,y),\cdots,\tilde{\theta}_d(x,y)),y,m).
    \end{split}
\end{equation}
We let 
$$
b(x,y,m):=\frac{\tilde{\chi}\left(\tilde{\theta}_1(x,y),\cdots,\tilde{\theta}_d(x,y)\right) \tilde{A}(\widetilde{x'}(\tilde{\theta}_1(x,y),\cdots,\tilde{\theta}_d(x,y)),y,m)}{\left(\det\frac{m\tilde{\Psi}''_{\tilde{\theta}\tilde{\theta}}(x,y,\tilde{\theta}_1(x,y),\cdots,\tilde{\theta}_d(x,y))}{2\pi i}\right)^{\frac{1}{2}}}\neq 0,
$$
and we can check that $b(x,y,m)\in S^{n-\frac{d}{2}}_{\rm loc}(1;D_p\times D_p)$. Moreover, for the asymptotic sum
$$
b(x,y,m)\sim\sum_{j=0}^\infty m^{n-\frac{d}{2}-j}b_j(x,y)~\text{in}~S^{n-\frac{d}{2}}_{{\rm loc}}(1;D_p\times D_p),
$$
we can find that for $p\in Y$, by construction of almost analytic functions, for some constant $C>0$,
\begin{equation}
    b_0(p,p)
    =\frac{1}{(2\pi)^{\frac{d}{2}}}{\frac{A_0(p,p)}{\left(\det\frac{{\Psi}''_{{\theta}{\theta}}(p,p,0,\cdots,0)}{i}\right)^{\frac{1}{2}}}}\\
    \geq C\frac{|\det\mathcal{L}_p|}{2^{\frac{d}{2}+1}\pi^{n+1+\frac{d}{2}}}>0.
\end{equation}
Here, we choose the branch as in Proposition \ref{MS formula} such that $\left(\det\frac{{\Psi}''_{{\theta}{\theta}}(p,p,0,\cdots,0)}{i}\right)^{\frac{1}{2}}>0$. To finish the proof of Theorem \ref{Main theorem 4}, it establish the followings:
\begin{proposition}
\label{Phase function property 0}
Let $f(x,y):=\tilde{\Psi}(x,y,\tilde{\theta}_1(x,y),\cdots,\tilde{\theta}_d(x,y))$, then ${\rm Im}f\geq 0$. Moreover,
$$
f(x,x)=0,~d_x f(x,x)=-\omega_0(x),~d_y f(x,x)=\omega_0(x),~b_0(x,x)>0~\text{for all}~x\in Y\cap D_p.
$$
\end{proposition}
\begin{proof}
First of all, by Proposition \ref{complex critical point} and \ref{MS formula}, ${\rm Im}f\geq 0$ holds. Second, for $p\in Y$ and a small BRT patch $D_p$ near $p$, by the construction of almost analytic extension and (\ref{function Phi under BRT}), we can find
\begin{equation}
    \label{f(p,p)=0}
    f(p,p)=\tilde{\Psi}(p,p,0)=\Phi(0,0)=0.
\end{equation}
By continuity, we may assume that $|f(x,y)|<\frac{1}{2}$ on $D_p\times D_p$ by taking $D_p$ small enough. Also, for $k=1,\cdots,2n+1$,
\begin{equation*}
        \frac{\partial f}{\partial x_k}(p,p)
        =\frac{\partial\tilde{\Psi}}{\partial x_k}(p,p,0)+\sum_{l=1}^d\frac{\partial\tilde{\Psi}}{\partial\tilde{\theta}_l}(p,p,0)\frac{\partial\tilde{\theta}}{\partial x_k}(p,p,0)
        =\frac{\partial\Phi}{\partial x_k}(0,0)
        =
        \begin{cases}
        0:k=1,\cdots,2n.\\
        1:k=2n+1.
        \end{cases}
\end{equation*}
and similarly 
\begin{equation*}
        \frac{\partial f}{\partial y_k}(p,p)
        =
        \begin{cases}
        0:k=1,\cdots,2n.\\
        -1:k=2n+1.
        \end{cases}
\end{equation*}
We can check that on $D_p$, $\omega_0(x)=-dx_{2n+1}+i\sum_{j=1}^n\left(\frac{\partial\phi}{\partial z_j}dz_j-\frac{\partial\phi}{\partial d\overline{z}_j}d\overline{z}_j\right)$, and hence 
\begin{equation}
    \label{df}
    d_x f(p,p)=-\omega_0(p),~d_y f(p,p)=\omega_0(p).
\end{equation}
Second, for all $x\in Y\cap D_p$, we now prove
\begin{equation*}
f(x,x)=0~\text{and}~d_x f(x,x)=-\omega_0(x),~d_y f(x,x)=\omega_0(x).
\end{equation*}
For $p\in Y$ and a small BRT patch 
$D_{p}$ near $p$, if we take any other $q\in Y\cap D_p$ and another small BRT patch $D_{q}$ near $q$, by the discussion in this section, we can write
$$
\Pi^{(0)}_{\leq\lambda,m,mp_1,\cdots,mp_d}(x,y)\equiv e^{im f(x,y)}b(x,y,m)~{\rm mod}~O(m^{-\infty})~\text{on }~D_{p}\times D_{p},
$$
where $f(p,p)=0$, $df(p,p)=0$, $|f(x,y)|<\frac{1}{2}$ on $D_p\times D_p$,
$$
b(x,y,m)\sim\sum_{j=0}^\infty m^{n-\frac{d}{2}-j}b_j(x,y)~\text{in}~S^{n-\frac{d}{2}}_{\rm loc}(1;D_{p}\times D_{p}),~b_j(x,y)\in\mathscr{C}^\infty(D_{p}\times D_{p}),~b_0(p,p)>0
$$
and
$$
\Pi^{(0)}_{\leq\lambda,m,mp_1,\cdots,mp_d}(x,y)\equiv e^{im f_1(x,y)}b_1(x,y,m)~{\rm mod}~O(m^{-\infty})~\text{on}~D_{q}\times D_{q},
$$
where $f_1(q,q)=0$, $df_1(q,q)=0$, $|f_1(x,y)|<\frac{1}{2}$ on $D_q\times D_q$,
$$
b_1(x,y,m)\sim\sum_{j=0}^\infty m^{n-\frac{d}{2}-j}(b_1)_j(x,y)~\text{in}~S^{n-\frac{d}{2}}_{\rm loc}(1;D_{q}\times D_{q}),~(b_1)_j(x,y)\in\mathscr{C}^\infty(D_{q}\times D_{q}),~(b_1)_0(q,q)>0.
$$
By consinuity, we may assume that $|b_0(x,y)|>0$ on $D_p\times D_p$. We arrange
\begin{equation}
\label{difference of kernel}
    e^{imf(x,y)}b(x,y,m)=e^{imf_1(x,y)}b_1(x,y,m)+R(x,y,m),~R(x,y,m)=O(m^{-\infty}).
\end{equation}
into
\begin{equation}
    e^{im(f-f_1)(x,y)}b(x,y,m)=b_1(x,y,m)+e^{-imf_1(x,y)}R(x,y,m),
\end{equation}
and after evaluating at the point $(x,y)=(q,q)$, we get
\begin{equation*}
    e^{im(f-f_1)(q,q)}b(q,q,m)=b_1(q,q,m).
\end{equation*}
Since
\begin{equation*}
    \lim_{m\to\infty} e^{-m{\rm Im}f(q,q)}=\lim_{m\to\infty}\frac{|b_1(q,q,m)|}{|b(q,q,m)|}=\frac{|(b_1)_0(q,q)|}{|b_0(q,q)|},
\end{equation*}
which is a non-zero finite number, we can conclude that ${\rm Im}f(q,q)=0$. Moreover, notice that
\begin{equation*}
    \lim_{m\to\infty}e^{imf(q,q)}=\lim_{m\to\infty}\frac{b_1(q,q,m)}{b(q,q,m)}=\frac{(b_1)_0(q,q)}{b_0(q,q)},
\end{equation*}
i.e.~the limit exists. However, the limit
\begin{equation*}
    \lim_{m\to\infty}e^{imf(q,q)}=\lim_{m\to\infty}e^{im{\rm Re}f(q,q)}=\lim_{m\to\infty}\left(\cos(m{\rm Re}f(q,q))+i\sin(m{\rm Re}f(q,q))\right)
\end{equation*}
does not exists if $|{\rm Re} f(q,q)|<\frac{1}{2}$, ${\rm Re}f(q,q)\neq 0$. Hence, we conclude that 
\begin{equation}
\label{f(x,x)=0}
    f(x,x)=0,~\text{for all}~x\in Y\cap D_p.
\end{equation}
Next, take derivative in both sides of (\ref{difference of kernel}) with respect to $x_j$,
$j=1,\cdots,2n+1$, and evaulate at $(x,y)=(q,q)$, then we have
\begin{equation*}
    im\frac{\partial}{\partial x_j}(f-f_1)(q,q)b(q,q,m)+\frac{\partial}{\partial x_j}b(q,q,m)=\frac{\partial}{\partial x_j}b_1(q,q,m)+\mathring{R}(x,y,m),
\end{equation*}
where $\mathring{R}(x,y,m):=\frac{\partial}{\partial x_j}R(q,q,m)-im\frac{\partial}{\partial x_j}f_1(q,q)R(q,q,m)=O(m^{-\infty})$. Therefore, for some constant $C>0$,
$$
\left|\frac{\partial}{\partial x_j}(f-f_1)(q,q)\right|=\lim_{m\to\infty}\frac{\left|\left(\frac{\partial}{\partial x_j}(b_1-b)+\mathring{R}\right)(q,q,m)\right|}{m|b(q,q,m)|}\leq\lim_{m\to\infty}\frac{C m^{n-\frac{d}{2}}}{|b_0(q,q)|m^{n-\frac{d}{2}+1}}=0
$$
Hence, for all $q\in Y\cap D_p$,
\begin{equation}
    \label{partial x f-f_1}
    \frac{\partial}{\partial x_j}\left(f-f_1\right)(q,q)=0,~j=1,\cdots,2n+1,
\end{equation}
Similarly, for all $q\in Y\cap D_p$,
\begin{equation}
    \label{partial y f-f_1}
    \frac{\partial}{\partial y_j}\left(f-f_1\right)(q,q)=0,~j=1,\cdots,2n+1.
\end{equation}
Combine (\ref{df}), (\ref{partial x f-f_1}) and (\ref{partial y f-f_1}), we establish
\begin{equation}
    d_x f(x,x)=-\omega_0(x),~d_y f(x,x)=\omega_0(x)~\text{for all}~x\in Y\cap D_p.
\end{equation}
In the last, from (\ref{difference of kernel}), by evaluating at the point $(x,y)=(q,q)$, we can find
$$
b(q,q,m)=b_1(q,q,m)+R(q,q,m),
$$
where $R(q,q,m)=O(m^{-\infty})$. Accordingly, 
$$
1=\lim_{m\to\infty}\frac{b(q,q,m)}{b_1(q,q,m)}=\frac{b_0(q,q)}{(b_1)_0(q,q)}.
$$
Thus, we can conclude that $b_0(q,q)=b_1(q,q)>0$, and since this holds for all $q\in Y\cap D_p$, we complete the proof of this proposition.
\end{proof}
Combine all the discussion in this section, the proof of Theorem \ref{Main theorem 3} is completed.
\section*{Acknowledgement}
The methods of microlocal analysis on CR manifolds with group action used in this work are marked by the influence of Professor Chin-Yu Hsiao and Professor George Marinescu. The author in particular wishes to express his hearty thanks to them for discussions on similar subjects and for giving the idea of the proof.

\bibliographystyle{amsxport}
\bibliography{reference}
\end{document}